\documentclass[a4paper, reqno]{amsart}

\usepackage{microtype}
\usepackage{amsmath,amstext,amssymb,mathrsfs,amscd,amsthm,indentfirst}
\usepackage{amsfonts}
\usepackage{mathtools}
\usepackage{stmaryrd}

\usepackage{booktabs}
\usepackage{verbatim}
\usepackage{enumerate}

\usepackage{graphicx}
\usepackage{bbm}

\numberwithin{equation}{section}

\usepackage{tikz, tikz-cd}
\usetikzlibrary{matrix,calc,positioning,arrows,decorations.pathreplacing,decorations.markings,patterns}
\tikzset{->-/.style={decoration={
			markings,
			mark=at position #1 with {\arrow{>}}},postaction={decorate}}}

\newcommand{\CircNum}[1]{\ooalign{\hfil\raise .00ex\hbox{\scriptsize #1}\hfil\crcr\mathhexbox20D}}

\newcommand{\bC}{\mathbb{C}}

\newcommand{\bF}{\mathbb{F}}

\newcommand{\bH}{\mathbb{H}}

\newcommand{\bN}{\mathbb{N}}
\newcommand{\bO}{\mathbb{O}}
\newcommand{\bP}{\mathbb{P}}
\newcommand{\bQ}{\mathbb{Q}}
\newcommand{\bR}{\mathbb{R}}

\newcommand{\bZ}{\mathbb{Z}}

\newcommand{\cB}{\mathcal{B}}

\newcommand\lra{\longrightarrow}

\newcommand\Diff{\mathrm{Diff}}

\newcommand{\hcoker}{/\!\!/}

\newcommand{\map}{\mathrm{map}}

\newcommand{\Fr}{\mathrm{Fr}}
\newcommand{\R}{\bR}

\newcommand{\Z}{\mathbbm{Z}}

\newcommand{\Oo}{\mathcal{O}}

\newcommand{\hAut}{\mathrm{hAut}}

\newcommand{\moddd}{/\!\!/}

\renewcommand{\epsilon}{\varepsilon}

\newcommand{\SO}{\mathrm{SO}}
\newcommand{\OO}{\mathrm{O}}
\newcommand{\Spin}{\mathrm{Spin}}
\newcommand{\U}{\mathrm{U}}

\newcommand{\cL}{\mathcal{L}}
\newcommand{\cM}{\mathcal{M}}
\newcommand{\cE}{\mathcal{E}}

\newcommand{\nin}{\not\in}
\newcommand{\orientationpreserving}{\mathrm{or}}
\newcommand{\Sing}{\mathrm{Sing}}
\newcommand{\GL}{\mathrm{GL}}
\newcommand{\bk}{\mathbbm{k}}

\newcommand{\str}{\rho}

\mathchardef\ordinarycolon\mathcode`\:
\mathcode`\:=\string"8000
\begingroup \catcode`\:=\active
  \gdef:{\mathrel{\mathop\ordinarycolon}}
\endgroup

\usepackage{amsthm}
\theoremstyle{plain}

\newtheorem{theorem}{Theorem}[section]

\newtheorem{proposition}[theorem]{Proposition}
\newtheorem{lemma}[theorem]{Lemma}

\newtheorem{corollary}[theorem]{Corollary}

\theoremstyle{definition}
\newtheorem{definition}[theorem]{Definition}

\newtheorem{example}[theorem]{Example}

\newtheorem{construction}[theorem]{Construction}

\theoremstyle{remark}

\newtheorem{remark}[theorem]{Remark}
\newtheorem*{remark*}{Remark}

\title[Moduli spaces of manifolds]{Moduli spaces of manifolds: a user's guide}

\author{S{\o}ren Galatius}
\email{galatius@math.ku.dk}
 \address{Department of Mathematics\\
   University of Copenhagen\\
   Denmark}

\author{Oscar Randal-Williams}
\email{o.randal-williams@dpmms.cam.ac.uk}
\address{Centre for Mathematical Sciences\\
Wilberforce Road\\
Cambridge CB3 0WB\\
UK}

\begin{document}
\begin{abstract}
  We survey recent work on moduli spaces of manifolds with an emphasis
  on the role played by (stable and unstable) homotopy theory.  The
  theory is illustrated with several worked examples.
\end{abstract}
\maketitle

\section{Introduction} 

The study of manifolds and invariants of manifolds was begun more
than a century ago.  In this entry we shall discuss the parametrised
setting: invariants of \emph{families} of manifolds, parametrised by a
base manifold $X$.  The invariants we look for will be cohomology
classes in $X$, characteristic classes.

This article will be structured in the following way.
\begin{enumerate}[(i)]
\item A discussion of the abstract classification theory.
  The main content here is the precise definition of the kind of
  families we consider, and an outline of how a classifying space may
  be constructed.  There is one such classifying space for each pair $(W,\rho_W)$ consisting of a closed manifold $W$ and a tangential structure $\rho_W$, see \S\ref{sec:smooth-bundles-their}.  The more general case where $W$ is compact with boundary is briefly discussed in \S\ref{sec:boundary}.
\item Definition of Miller--Morita--Mumford classes, the main examples of
  characteristic classes.  With this definition in place, we state a first version of the main result of \cite{GR-W2, GR-W3, GR-W4}: a formula (see Theorem~\ref{thm:main-cohomological}) for rational cohomology of the classifying spaces in a range of degrees, for many instances of $(W,\rho_W)$.
\item Statement of the main result of \cite{GR-W2, GR-W3, GR-W4} in their general forms.  These statements require a bit more homotopy theory to formulate but apply quite generally, at least for even-dimensional manifolds, see the theorems in \S\ref{sec:gener-vers-main}.
\item Examples of calculations and applications.  The results surveyed in \S\S\ref{sec:char-class-XXX}--\ref{sec:gener-vers-main} are quite well suited for explicit calculations.  We believe this to be an important feature of the theory, and have included a supply of worked examples of various types to illustrate this aspect, many of which have not previously been published.  In \S\ref{sec:RatCalc} we will focus on calculations in rational cohomology, and we include a detailed study of some complete intersections.  In \S\ref{sec:AbCalc} we carry out some integral homology calculations, focusing on $H_1$.
\end{enumerate}

Our theorems apply in even dimensions $2n \geq 6$, but were inspired by the breakthrough theorem of Madsen and Weiss \cite{MW} in dimension $2$, building on earlier ideas of Madsen and Tillmann \cite{MT} and Tillmann \cite{Tillmann}.

\section{Smooth bundles and their classifying spaces}
\label{sec:smooth-bundles-their}

First, some conventions.  By \emph{smooth manifold}\index{manifold} we shall always mean a Hausdorff, second countable topological manifold equipped
with a maximal $C^\infty$ atlas, and \emph{smooth map} shall always mean
$C^\infty$. We shall generally use the letter $\mathcal{M}$
  with various decorations for variants of classifying spaces for families of manifolds,
  and the letter $\mathcal{F}$ with various decorations for the
  functor it classifies.

\subsection{Smooth bundles}
\label{sec:smooth-bundles}

Our notion of ``family'' of manifolds will be \emph{smooth fibre
  bundle}, possibly equipped with extra tangent bundle structure, as
follows.  If $V$ is a $d$-dimensional real vector bundle we shall
write $\mathrm{Fr}(V)$ for the associated frame bundle, which is a
principal $\GL_d(\R)$-bundle.

\begin{definition}\label{def:family} Let $d$ be a non-negative integer.
  \begin{enumerate}[(i)]
  \item A \emph{smooth fibre bundle}\index{fibre bundle} of dimension $d$ consists of smooth manifolds $E$
    and $X$ (without boundary) and a smooth proper map $\pi: E\to X$
    such that $D\pi: TE \to \pi^* TX$ is surjective and the vector
    bundle $T_\pi E = \mathrm{Ker}(D\pi)$ has $d$-dimensional fibres. The bundle $T_\pi E$ is called the \emph{vertical tangent bundle}.\index{vertical tangent bundle}
  \item If $\Theta$ is a space with a continuous action of
    $\GL_d(\R)$, a \emph{smooth fibre bundle with $\Theta$-structure}\index{$\Theta$-structure}
    consists of a smooth fibre bundle $\pi: E \to X$, together with a continuous
    $\GL_d(\R)$-equivariant map
    $\str: \mathrm{Fr}(T_\pi E) \to \Theta$.
  \end{enumerate}
\end{definition}

Typical choices of $\Theta$ include the terminal one $\Theta = \{\ast\}$,
and $\Theta = \Z^\times = \{\pm 1\}$ on which $\GL_d(\R)$ acts by multiplication by
the sign of the determinant.  In the former case a $\Theta$-structure
is no information, and in the latter it is the data of a continuously
varying family of orientations of the $d$-manifolds $\pi^{-1}(x)$.  (The space of equivariant maps $\mathrm{Fr}(T_\pi E) \to \Theta$ may be modelled in other ways, equivalent up to weak equivalence, see \S\ref{sec:gl_dr-spaces-versus}.)

Any smooth map $f: X' \to X$ will be transverse to any smooth fibre bundle
$\pi: E \to X$ as above, and the pullback $(f^* \pi): f^* E \to X'$ is
again a smooth fibre bundle.  Given a $\Theta$-structure $\str$ on
$\pi$, we shall write $f^*\str$ for the induced structure on
$(f^* \pi)$.

\subsection{Classifying spaces}
\label{sec:classifying-spaces}

\newcommand{\class}{\mathcal{F}}
\newcommand{\Man}{\mathsf{Man}}
\newcommand{\sSets}{\mathsf{sSets}}
\newcommand{\op}{\mathrm{op}}

The natural equivalence relation between the bundles considered above
is \emph{concordance}, which we recall.
\begin{definition} Let $\pi_0: E_0 \to X$ and $\pi_1: E_1 \to X$ be
  smooth bundles with $\Theta$-structures
  $\str_0: \Fr(T_{\pi_0} E_0) \to \Theta$ and
  $\str_1: \Fr(T_{\pi_1} E_1) \to \Theta$.
  \begin{enumerate}[(i)]
  \item An \emph{isomorphism} between $(\pi_0,\str_0)$ and
    $(\pi_1,\str_1)$ is a diffeomorphism $\phi: E_0 \to E_1$ over $X$,
    such that the induced map
    $\mathrm{Fr}(D_\pi \phi): \Fr(T_{\pi_0}E_0) \to \Fr(T_{\pi_1}E_1)$
    is over $\Theta$.
  \item A \emph{concordance}\index{concordance} between $(\pi_0,\str_0)$ and
    $(\pi_1,\str_1)$ is a smooth fibre bundle $\pi: E \to \R \times X$
    with $\Theta$-structure $\str: \Fr(T_\pi E) \to \Theta$, together
    with isomorphisms from $(\pi_0,\str_0)$ and $(\pi_1, \str_1)$ to
    the pullbacks of $(\pi,\str)$ along the two embeddings
    $X \cong \{0 \} \times X \subset \R\times X$ and
    $X \cong \{1 \} \times X \subset \R\times X$.
  \end{enumerate}
\end{definition}

Pulling back is functorial up to isomorphism and preserves being concordant.
\begin{definition}
  For a smooth manifold $X$ without boundary, let $\class^\Theta[X]$
  denote the set of concordance classes of pairs $(\pi,\str)$ of a
  smooth fibre bundle $\pi: E \to X$ with $\Theta$-structure
  $\str: \mathrm{Fr}(T_\pi E) \to \Theta$.  (Note that $X$ is fixed,
  but $E$ is allowed to vary.)
\end{definition}

\begin{theorem}\label{thm:existence-of-classi}
  The functor $X \mapsto \class^\Theta[X]$ is representable in the
  (weak)  sense that there exists a topological space
  $\mathcal{M}^\Theta$\index{moduli space!of $\Theta$-manifolds} and a natural bijection
  \begin{equation}\label{eq:3}
    \class^\Theta[X] \cong [X,\mathcal{M}^\Theta],
  \end{equation}
  where the codomain denotes homotopy classes of continuous maps.
\end{theorem}

There are various ways to prove this representability statement.  In
the series \cite{GR-W2, GR-W3, GR-W4} we did this by constructing explicit point-set
models in terms of submanifolds of $\R^\infty$.  Instead of repeating
what we said there, let us outline an approach based on simplicial sets and the
observation that the functor $\class^\Theta$ may be upgraded to take
values in \emph{spaces}, and the natural bijection~(\ref{eq:3}) may be
upgraded to a natural weak equivalence to the mapping
space.
\begin{definition}\label{defn:represented-functor}
  Let $\Delta^\bullet_e$ be the cosimplicial smooth manifold given by
  the extended simplices
  $\Delta^p_e = \{t \in \R^{p+1} \mid t_0 + \dots + t_p = 1\}$.

  For a smooth manifold $X$ let $\class^\Theta_\bullet(X)$ denote the
  simplicial set whose $p$-simplices are all pairs $(\pi,\str)$
  consisting of a smooth fibre bundle $\pi: E \to \Delta^p_e \times X$
  with $\Theta$-structure $\str: \Fr(T_\pi E) \to \Theta$.
\end{definition}

As they stand, these definitions are not quite rigorous:
$\class^\Theta[X]$ and $\class^\Theta_p(X)$ are not (small) sets for
size reasons, and $\class^\Theta_p(X)$ is not functorial in
$[p] \in \Delta$ because pullback is not strictly associative on the
level of underlying sets.  This may be fixed in standard ways, e.g.\
by requiring the underlying set of $E$ to be a subset of
$X \times \Omega$ for a set $\Omega$ of sufficiently large
cardinality.  See e.g.\ \cite[Section 2.1]{MW} for more
detail.
\begin{proposition}
  Let $\Man$ and $\sSets$ be the categories of smooth manifolds and
  simplicial sets, and let
  $\class^\Theta_\bullet: \Man^\op \to \sSets$ be the functor defined
  above.  Then we have a natural bijection
  \begin{equation*}
    \class^\Theta[X] \cong \pi_0 \class^\Theta_\bullet (X)
  \end{equation*}
  and a natural weak equivalence of simplicial sets
  \begin{equation}\label{eq:4}
    \class^\Theta_\bullet(X) \xrightarrow\simeq
    \mathrm{Maps}(\mathrm{Sing}(X),\class_\bullet(\{\ast\})).
  \end{equation}
  Consequently, we may take $\mathcal{M}^\Theta := |\class^\Theta_\bullet(\{\ast\})|$ in Theorem~\ref{thm:existence-of-classi}.
\end{proposition}
In this statement we  take $\mathrm{Sing}(X)$ to mean the smooth singular set, i.e.\ the simplicial set $[p] \mapsto C^\infty(\Delta^p_e,X)$.  It is equivalent to the usual simplicial set made out of all continuous maps $\Delta^p \to X$ by a smooth approximation argument, and in particular the evaluation map $|\mathrm{Sing}(X)| \to X$ is a homotopy equivalence by Milnor's theorem.  The codomain of~(\ref{eq:4}) is the simplicial set of maps into the Kan complex $\class_\bullet(\{\ast\})$, homotopy equivalent to the space of maps $X \to |\class_\bullet(\{\ast\})|$.
\begin{proof}[Proof sketch]
  The first claim follows by identifying 1-simplices of
  $\class^\Theta_\bullet(X)$ with concordances between 0-simplices.

  For the second claim, we define the
  map~(\ref{eq:4}) by pulling back.  To check that it
  is a weak equivalence one first verifies that both sides send open
  covers $X = U \cup V$ to homotopy pullback squares and countable
  increasing unions to homotopy limits.  Hence it suffices to check the
  case $X = \R^n$.  Then one checks that both sides send
  $X \times \R \to X$ to a weak equivalence, so it suffices to check
  $X = \{\ast\}$, which is obvious.

  The third claim follows by combining the first and the second claim and the homotopy equivalence $|\Sing(X)| \to X$ given by evaluation.  Alternatively it may be quoted directly from \cite[Section 2.4]{MW}.
\end{proof}

The above proof of Theorem~\ref{thm:existence-of-classi} gives an explicit bijection~(\ref{eq:3}).  Indeed, an element $(\pi,\str) \in \class_0(X)$ gives a morphism of simplicial sets $\Sing(X) \to \class^\Theta_\bullet(\{\ast\})$ and hence a canonical zig-zag
\begin{equation}\label{eq:24}
  X \overset{\mathrm{ev}}\longleftarrow |\Sing(X)| \overset{(\pi,\str)}\lra | \class^\Theta_\bullet(\{\ast\})| = \cM^\Theta.
\end{equation}
The map $\mathrm{ev}: |\Sing(X)| \to X$ is a homotopy equivalence, and we may choose a homotopy inverse.  For example, a smooth triangulation of $X$ gives such an inverse, which is even a section.  The resulting homotopy class of map $X \to \cM^\Theta$ corresponds to $[(\pi,\str)]$ under the bijection~(\ref{eq:3}).

For later purposes, let us explain how a universal fibration $\pi^\Theta : \cE^\Theta \to \cM^\Theta$ modelling the bundle $\pi: E \to X$ may be constructed by the same simplicial method.  Let $\widetilde{\class}^\Theta_\bullet(X)$ be the simplicial set whose $p$-simplices are triples $(\pi,\str,s)$ where $(\pi,\str)$ is as before and $s: \Delta^p_e \times X \to E$ is a section of $\pi$, and set $\cE^\Theta := |\widetilde{\class}^\Theta_\bullet(\{\ast\})|$.  The zig-zag~(\ref{eq:24}) above associated to an element $(\pi,\str) \in \class_0^\Theta(X)$ now extends to a canonical diagram
\begin{equation*}
\begin{tikzcd}
  E \dar[']{\pi} & {|\Sing(E)|} \lar[']{\simeq}\dar\rar & \cE^\Theta \dar{\pi^\Theta}\\
  X & {|\Sing(X)|} \lar[']{\simeq} \rar{(\pi,\str)} & \cM^\Theta
\end{tikzcd}
\end{equation*}
in which both squares are homotopy cartesian, and the top right-hand map is that associated to the data $(\mathrm{pr}_1 : E \times_X E \to E, \str\circ D\mathrm{pr}_2, \mathrm{diag} : E \to E \times_X E)$.

Finally, a $p$-simplex $(\pi: E \to \Delta^p_e, \str, s)$ of $\cE^\Theta$ gives a map $\ell = (\str/\GL_d(\bR)): E \cong (\Fr(T_\pi E)/\GL_d(\R)) \to (\Theta/\GL_d(\R))$.  Composing with the section $s\vert_{\Delta^p}: \Delta^p \to E$ then gives rise to a map of simplicial sets 
$$\widetilde{\class}^\Theta_\bullet(\{*\}) \lra \Sing(\Theta/\GL_d(\R))$$
realising to a map $\cE^\Theta \to |\Sing(\Theta/\GL_d(\R))|$. The orbit space $\Theta/\GL_d(\R)$ need not be well behaved, and we would like to replace it by the homotopy orbit space $B = \Theta \moddd \GL_d(\R) := (E\GL_d(\R) \times \Theta)/\GL_d(\R)$. Taking homotopy orbits of the map $\Theta \to \{*\}$ yields a map $\theta : B \to B\GL_d(\R)$. Repeating the construction above with $E\GL_d(\R) \times \Theta$ instead of $\Theta$ allows us to construct a zig-zag
\begin{equation*}
  \cE^\Theta \overset{\simeq}\longleftarrow \cE^{E \GL_d(\R) \times \Theta} \overset{\ell}\lra |\Sing(B)| \overset{\mathrm{ev}}\lra B \overset{\theta}\lra B\GL_d(\R).
\end{equation*}

Slightly less precisely, we shall summarise this situation as a diagram
\begin{equation}\label{eq:25}
  \begin{tikzcd}
    E \dar[']{\pi} \rar & \cE^\Theta \dar{\pi^\Theta} \rar{\ell} & B \rar{\theta}& B\GL_d(\R)
    \\
    X \rar &  \cM^\Theta
  \end{tikzcd}
\end{equation}
where the square is homotopy cartesian. The vector bundle classified by the composition $\theta\circ\ell: \cE^\Theta \to B\GL_d(\R)$ is a universal instance of $T_\pi E$, and shall be denoted $T_\pi \cE^\Theta$.  The factorisation through $\theta$ gives an equivariant map $\str^\Theta : \Fr(T_\pi \cE^\Theta) \to \Theta$, a universal instance of $\str: \Fr(T_\pi E) \to \Theta$.

\subsection{Path connected classifying spaces}

The bundles classified so far are inconveniently general.  For
example, we have not made any restrictions on the diffeomorphism type
of the fibres of $\pi: E \to X$.  Hence, if e.g.\ $\Theta = \{\ast\}$,
the set $\pi_0(\mathcal{M}^\Theta)$ is in bijection with the set of
diffeomorphism classes of compact smooth $d$-manifolds without
boundary, a countably infinite set for any $d \geq 0$.  The homotopy
type of $\mathcal{M}^\Theta$ then encodes at once the classification of 
smooth manifolds up to diffeomorphism and the classification of smooth
bundles, because it is the ``moduli space'' (or ``$\infty$-groupoid'')
of \emph{all} $d$-manifolds.

We shall often study one path component of
$\mathcal{M}^\Theta$ at a time, which corresponds to fixing the
concordance class of the fibres of the classified bundles.  We
introduce the following notation.
\begin{definition}
  Let $\Theta$ be a space with $\GL_d(\R)$ action, $W$ be a
  compact $d$-manifold without boundary, and
  $\str_W: \Fr(TW) \to \Theta$ be an equivariant map.  Considering this as a family over a point yields a $[(W,\str_W)] \in \pi_0 \mathcal{M}^\Theta$, and we shall write
  $\mathcal{M}^\Theta(W,\str_W) \subset \mathcal{M}^\Theta$\index{moduli space!of $\Theta$-manifolds} for the
  path component containing $(W,\str_W)$.  This path component is a classifying space
  for smooth fibre bundles $\pi: E \to X$ with structure
  $\str: \Fr(T_\pi E) \to \Theta$, whose restriction to any point
  $\{x\} \subset X$ is concordant to $(W,\str_W)$.
\end{definition}

\begin{remark}\label{remark:special-cases}
  In the special case $\Theta = \{\ast\}$ the maps $\str: \Fr(T_\pi E) \to \Theta$ are irrelevant, and in this case we shall write simply $\cM(W)$.  This space classifies smooth bundles $\pi: E \to X$ with fibres diffeomorphic to $W$ with no further structure, and we have the weak equivalence
  \begin{equation*}
    \cM(W) \simeq B\Diff(W),
  \end{equation*}
  where $\Diff(W)$ is the diffeomorphism group\index{diffeomorphism group} of $W$ in the $C^\infty$ topology.  Similarly, if $\Theta = \Z^\times$ with the orientation action and $W$ is given an orientation $\str_W$, we shall write $\cM^\orientationpreserving(W)$ for $\cM^\Theta(W,\str_W)$ and there is a weak equivalence
  \begin{equation*}
    \cM^\orientationpreserving(W) \simeq B\Diff^\orientationpreserving(W),
  \end{equation*}
  where $\Diff^\orientationpreserving(W) \subset \Diff(W)$ is the subgroup of orientation preserving diffeomorphisms.  For general $\Theta$, the homotopy type is described as the Borel construction
   \begin{equation*}
    \mathcal{M}^\Theta(W,\str_W) \simeq
    \frac{\{\text{$\str: \Fr(TW) \to \Theta$, equivariantly homotopic to
        $\str_W$}\}}
      {\left(\substack{\text{topological group of diffeomorphisms $W
              \to W$} \\ \text{preserving the
          equivariant homotopy
          class of $\str_W$}}\right)}.
  \end{equation*}
   See \cite[Definition 1.5]{GR-W2}, \cite[Section 1.1]{GR-W2}, or \cite[Section 1.1]{GR-W4} for further discussion of this point of view. 
\end{remark}

\section{Characteristic classes}
\label{sec:char-class-XXX}  

The main topic of this article is the study of characteristic classes
of the sort of bundles described above, i.e.\ the calculation of the
cohomology ring of the classifying spaces
$\mathcal{M}^\Theta(W,\str_W)$.  We first recall the conclusions in
rational cohomology, which are easier to state and often gives an
explicit {formula} for the ring of characteristic classes.

\subsection{Characteristic classes}
\label{sec:char-class-1}

As before, let $B = \Theta \moddd \GL_d(\R)$ denote the Borel construction.  For a smooth bundle $\pi: E \to X$ with $\Theta$-structure
$\str: \Fr(T_\pi E) \to \Theta$, we may form the Borel construction
with $\GL_d(\R)$ and obtain a correspondence, i.e.\ a diagram
\begin{equation}\label{eq:1}
  X \overset{\pi}\longleftarrow E \overset{\simeq}\longleftarrow \Fr(T_\pi E)\moddd
  \GL_d(\R) \xrightarrow{\str\moddd \GL_d(\R)} B.
\end{equation}
We write $\ell: E \to B$ for the homotopy class of maps associated to $\str\moddd \GL_d(\R)$. We shall let $\Z^{w_1}$ denote the coefficient system on $B$ arising from
the non-trivial action of $\pi_0(\GL_d(\R)) = \Z^\times$ on $\Z$, and
let $A^{w_1} = A \otimes \Z^{w_1}$ for an abelian group $A$, and we shall use the same notation for these coefficient systems pulled back to $E$ along~(\ref{eq:1}).  Then we have a homomorphism
\begin{equation}
  \label{eq:20}
  \ell^*: H^{k+d}(B;A^{w_1}) \lra H^{k+d}(E;A^{w_1}).
\end{equation}
We also have a
\emph{fibre integration}\index{fibre integration} homomorphism
\begin{equation}
  \label{eq:17}
  \int_\pi: H^{k+d}(E;A^{w_1}) \lra H^k(X;A),
\end{equation}
where $d$ is the dimension of the fibres of $\pi$.  It may be
defined e.g.\ as the composition $H^{k+d}(E;A^{w_1}) \twoheadrightarrow E_\infty^{k,d} \subset E_2^{k,d}$ in the Serre spectral sequence for the fibration $\pi$, or by a Pontryagin--Thom\index{Pontryagin--Thom construction} construction as in \S\ref{sec:mmm-class-gener} below.

Given a class $c \in H^{d + k}(B;A^{w_1})$, we may combine~(\ref{eq:20}) and~(\ref{eq:17}) and define the \emph{Miller--Morita--Mumford classes}\index{Miller--Morita--Mumford class}\index{MMM-class} (or MMM-classes) by the push-pull formula
\begin{equation*}
  \kappa_c(\pi) = \int_\pi \ell^*c \in H^k(X;A).
\end{equation*}
This class is easily seen to be natural with respect to pullback along
smooth maps $X' \to X$, and in fact comes from a universal class
\begin{equation*}
  \kappa_c \in H^k(\mathcal{M}^\Theta(W,\str_W);A)
\end{equation*}
for any $(W,\str_W)$, any $A$, and any $c \in H^{k+d}(B;A^{w_1})$, defined by the analogous push-pull formula in the universal instance~(\ref{eq:25}).

In the special case $k = 0$ and $X = \{\ast\}$, the definition of
$\kappa_c \in H^0(X;A) = A$ simply reproduces the usual
\emph{characteristic numbers},\index{characteristic number} c.f.\ \cite[\S16]{MS}.  For example, if $d = 2n$ and $e$ comes from the Euler class in $H^d(B\GL_d(\bR);\bZ^{w_1})$, then $\kappa_e \in H^0(X;\bZ) = \bZ$ is the Euler characteristic of $W$. Henceforth we shall mostly be interested in the case $k > 0$.

If $A=\bk$ is a field, we may combine with cup product to get a map of graded rings
\begin{equation}\label{eq:2}
  \bk[\kappa_c \mid c \in \text{basis of $H^{>d}(B;\bk^{w_1})$}] \lra H^*(\mathcal{M}^\Theta(W,\str_W);\bk),
\end{equation}
whose domain is the free graded-commutative $\bk$-algebra on symbols $\kappa_c$, one for each element $c$ in a chosen basis (or more invariantly, on a degree-shifted copy of the graded $\bk$-vector space $H^{>d}(B;\bk^{w_1})$).

\subsection{Genus}\label{sec:genus}

All of our results will hold in a range of degrees depending on \emph{genus}, a numerical invariant that we first introduce.

Assume that $d = 2n$ and let a $\GL_{2n}(\R)$-space $\Theta$ be
given.
We shall assume that $n > 0$ and that the homotopy orbit space
$\Theta\moddd \GL_{2n}(\R)$ is connected, i.e.\ that
$\pi_0(\GL_{2n}(\R)) = \Z^\times$ acts transitively on
$\pi_0(\Theta)$.

The genus will be defined in terms of the manifold obtained from $S^n \times S^n$ by removing a point, which plays a special role in this theory.  Up to diffeomorphism this manifold may be obtained
as a pushout
\begin{equation}\label{eq:26}
S^n \times \R^n \hookleftarrow \R^n \times \R^n \hookrightarrow \R^n
\times S^n,
\end{equation}
where the embeddings are induced by a choice of
coordinate chart $\R^n \hookrightarrow S^n$.  Following
\cite[Definition 1.3]{GR-W3} a $\Theta$-structure on $S^n \times \R^n$ shall be
called \emph{admissible} if it ``bounds a disk'', i.e.\ is (equivariantly) homotopic
to a structure that extends over some embedding
$S^n \times \R^n \hookrightarrow \R^{2n}$.  Note that this is automatic if $\pi_n(\Theta) = 0$ for some basepoint.
A structure on
$S^n \times S^n \setminus \{\ast\}$ is admissible if the
restriction to each piece of the gluing~(\ref{eq:26}) is admissible.

\begin{definition}
  Assume $d = 2n > 0$ and that $W$ is connected.  The \emph{genus}\index{genus} $g(W, \str_W)$
  of a $\Theta$-manifold $(W,\str_W)$ is the maximal number of disjoint embeddings
  $j: S^n \times S^n \setminus \{\ast\} \to W$ such that $j^*\str_W$ is
  admissible.
\end{definition}
This appeared as \cite[Definition 1.3]{GR-W3}.  For example, when
$n = 1$ and $\Theta = \pi_0(\GL_2(\R)) = \bZ^\times$, this is
precisely the usual genus of an oriented connected 2-manifold.  The admissibility condition may be illustrated by the case $\Theta = \GL_2(\R)$, corresponding to framings on 2-manifolds.  The Lie group framing on $\Sigma = S^1 \times S^1$ satisfies  $g(\Sigma,\rho_\mathrm{Lie}) = 0$, but there exist other framings $\rho$ for which $g(\Sigma,\rho) = 1$.

In \S\ref{sec:disc-moduli-spac} below we shall explain how to determine a lower bound on the number $g(W,\str_W)$ in terms of more accessible invariants.

\subsection{Main theorem in rational cohomology}
\label{sec:main-theor-rati}

The main results of \cite{GR-W2, GR-W3, GR-W4} imply that the ring
homomorphism homomorphism~(\ref{eq:2}) is often an isomorphism in a
range of degrees when $\bk = \bQ$.
We explain the statement.

\begin{definition}
  Assume $\Theta\moddd \GL_{2n}(\R)$ is connected.  We say $\Theta$ is
  \emph{spherical}\index{$\Theta$-structure!spherical} if $S^{2n}$ admits a $\Theta$-structure, i.e.\ if
  there exists an equivariant map $\Fr(TS^{2n}) \to \Theta$.
\end{definition}
The condition is equivalent to existence of an $\OO(2n)$-equivariant map
$\OO(2n+1) \to \Theta$.  This obviously holds if the $\GL_{2n}(\R)$-action on
$\Theta$ admits an extension to an action of $\GL_{2n+1}(\R)$ which
holds in many cases, e.g.\ $\Theta = \{\pm 1\}$ with the orientation
action.  See \cite[Section 5.1]{GR-W2} for more information about this
condition.

\begin{theorem}\label{thm:main-cohomological}
  Let $d = 2n > 4$, $W$ be a closed simply-connected $d$-manifold, and $\str_W: \Fr(W) \to \Theta$ be a $\Theta$-structure which is $n$-connected (or, equivalently, such that the associated $\ell_W : W \to B = \Theta \moddd \GL_d(\bR)$ is $n$-connected). Equip $B$ with the local system $\bQ^{w_1}$ as above, and assume that $H^{k+d}(B;\bQ^{w_1})$ is finite dimensional for each $k \geq 1$. Then the ring homomorphism
  \begin{equation*}
      \bQ[\kappa_c \mid c \in \text{basis of $H^{>d}(B;\bQ^{w_1})$}] \lra H^*(\mathcal{M}^\Theta(W,\str_W);\bQ),
  \end{equation*}
  as in~(\ref{eq:2}) is an isomorphism in cohomological degrees $\leq (g(W, \str_W)-4)/3$. If in addition $\Theta$ is spherical, then the range may be improved
  to $ \leq (g(W, \str_W)-3)/2$.
\end{theorem}

\subsection{Estimating genus}
\label{sec:disc-moduli-spac}

To apply Theorem~\ref{thm:main-cohomological} we must calculate the invariant $g(W,\str_W)$, or at least be able to give useful lower bounds for it.  This section explains a general such lower bound, under the assumptions that $d = 2n > 4$, and that the homotopy orbit space $B = \Theta \moddd \GL_d(\bR)$ is simply-connected.  We may then choose an equivariant map $\Theta \to \Z^\times$ by which any $\Theta$-structure induces an orientation, and we shall assume given such a map.

There is an obvious upper bound for $g(W,\str_W)$: it is certainly no
larger than the number of hyperbolic summands in $H_n(W;\Z)$ equipped
with its intersection form.  For odd $n$ this is in turn no more than
$b_n/2$ and for even $n$ it is no more than $\min(b_n^+,b_n^-)$, where
we write $b_n$ for the middle Betti number of $W$ and in the even case,
$b_n = b_n^+ + b_n^-$ for its splitting into positive and negative
parts.  More usefully, \cite[Remark 7.16]{GR-W3} gives the following \emph{lower}
bound on genus.

\begin{theorem}
  Assume $d = 2n > 4$, that the homotopy orbit space
  $B = \Theta \moddd \GL_d(\R)$ is simply-connected, and that
  $\ell_W: W \to B$ (or equivalently $\str_W: \Fr(TW) \to \Theta$) is
  $n$-connected.  Write $g^a(W) = \min(b_n^+,b_n^-)$ for $n$
  even and $g^a(W) = b_n/2$ for $n$ odd.  Then
  \begin{equation}\label{eq:15}
    g^a(W) - c \leq g(W,\str_W) \leq g^a(W),
  \end{equation}
  with $c = 1+e$, where $e$ is the minimal number of generators of
  the abelian group $H_n(B;\Z)$.  If $n$ is even or if $n \in \{3,7\}$ then one may take
  $c = e$.
\end{theorem}

\begin{remark}
  Let us briefly point out that the estimate~(\ref{eq:15}) may be
  expressed using characteristic numbers. Indeed, writing $b_i = b_i(B) = b_i(W)$ for $i = 1, \dots, n-1$, we have
  \begin{equation}\label{eq:18}
    g^a(W) = (-1)^n\big(\chi(W)/2 - \sum_{i = 0}^{n-1} (-1)^i b_i\big) - |\sigma(W)|/2,
  \end{equation}
  where $\chi(W) = \int_W e(TW)$ is the Euler characteristic and
  $\sigma(W) = \int_W \mathcal{L}(TW)$ is the signature (where we
  write $\sigma(W) = 0$ when $n$ is odd).
\end{remark}

\section{General versions of main results}
\label{sec:gener-vers-main}

For some purposes, the rational cohomology statement in
Theorem~\ref{thm:main-cohomological} suffices, but there are several homotopy theoretic strengthenings and variations given in \cite{GR-W4}, which we now explain. 

\subsection{Stable homotopy enhancement}

To state a more robust version of Theorem~\ref{thm:main-cohomological}
above, we must first introduce a space $\Omega^\infty MT\Theta$
associated to the $\GL_d(\R)$-space $\Theta$. The map $B = \Theta\moddd \GL_d(\R) \to B\GL_d(\R)$ classifies a
$d$-dimensional real vector bundle over $B$, and we shall write
$MT\Theta$ for the Thom spectrum of its virtual inverse and
$\Omega^\infty MT\Theta$ for the corresponding infinite loop
space.  The following result is a restatement of \cite[Corollary 1.7]{GR-W4}.

\begin{theorem}\label{thm:homotopical}
  Let $d = 2n > 4$, $W$ be a closed simply-connected $d$-manifold, and $\str_W: \Fr(W) \to \Theta$ be an
  $n$-connected equivariant map. Then there is a map
  \begin{equation}\label{eq:5}
    \alpha : \mathcal{M}^\Theta(W,\str_W) \lra \Omega^\infty MT\Theta,
  \end{equation}
  inducing an isomorphism in integral homology onto the path component that it hits, in degrees $\leq (g(W, \str_W) - 4)/3$.

  If in addition $\Theta$ is spherical, then~(\ref{eq:5}) induces an
  isomorphism in homology in degrees $\leq (g(W, \str_W)-3)/2$.
\end{theorem}
The statement proved in \cite{GR-W4} is in fact stronger: the map $\alpha$ induces an isomorphism in homology with local coefficients in degrees up to $(g(W, \str_W) - 4)/3$.  We say that the map is \emph{acyclic} in this range of degrees.  (The induced map in
$\pi_1$ is of course far from an isomorphism.)

\begin{remark}\label{rem:RatCalc}
  Let us also briefly recall why
  Theorem~\ref{thm:main-cohomological} is a consequence of
  Theorem~\ref{thm:homotopical}: under the Thom isomorphism $H^{k+d}(B;\bQ^{w_1}) \cong H^k(MT\Theta;\bQ)$ each class $c \in H^{k+d}(B;\bQ^{w_1})$ may be
  represented by a spectrum map $MT\Theta \to \Sigma^k H\bQ$.  If we
  choose a rational basis
  $\mathcal{B}_k \subset H^{k+d}(B;\bQ^{w_1}) \cong H^k(MT\Theta;\bQ)$ and represent each basis element by a spectrum map $MT\Theta \to \Sigma^k H\bQ$, we obtain
  \begin{equation*}
    MT\Theta \lra \prod_{k = 1}^\infty \prod_{c \in \mathcal{B}_k} \Sigma^k H\bQ
  \end{equation*}
  which induces isomorphisms in rational homology and hence in
  rationalised homotopy in positive degrees, at least if each $\mathcal{B}_k$ is finite in which case the product in the codomain may be replaced by the wedge.
  It follows that the induced map of infinite loop spaces
  \begin{equation*}
    \Omega^\infty MT\Theta \lra \prod_{k = 1}^\infty \prod_{\mathcal{B}_k} K(\bQ,k)
  \end{equation*}
  induces an equivalence in rationalised homotopy groups in positive
  degrees, hence in rational cohomology when restricted to any path
  component of its domain.
\end{remark}

\subsection{MMM-classes in generalised cohomology}\label{sec:mmm-class-gener}

There is a preferred map~(\ref{eq:5}) in Theorem
\ref{thm:homotopical}. Following \cite{MT}, in \cite[Remark 1.11]{GR-W2} we gave an
explicit point-set model.  Here we shall explain the map in a conceptual but somewhat
informal way.  It is obtained from the following two ingredients.
\begin{enumerate}[(i)]
\item A smooth $d$-manifold $W$ and an equivariant map
  $\str_W: \Fr(TW) \to \Theta$ induces a continuous map
$$\ell_W = \str_W\moddd\GL_d(\R): W \simeq \Fr(TW)\moddd\GL_d(\R) \lra B = \Theta
  \moddd \GL_d(\R)$$
under which the canonical bundle $\gamma$ on
  $B\GL_d(\R) = \ast\moddd\GL_d(\R)$ is pulled back to $TW$.  By
  passing to Thom spectra of inverse bundles one gets a map
  \begin{equation}\label{eq:9}
    W^{-TW} \lra B^{-\gamma} = MT\Theta,
  \end{equation}
  in the stable homotopy category.
\item If $W$ is a closed manifold there is a canonical map
  \begin{equation}\label{eq:8}
    S^0 \lra W^{-TW}
  \end{equation}
  which is Spanier--Whitehead dual\index{duality!Spanier--Whitehead} to the canonical map $W \to \{\ast\}$ under Atiyah duality $D(W_+) \simeq W^{-TW}$. The actual map of spectra depends on certain choices, but in the right
  setup these choices form a contractible space.  (For example, one may choose
  an embedding $W \hookrightarrow \R^\infty$ and get the map~(\ref{eq:8}) by
  the Pontryagin--Thom\index{Pontryagin--Thom construction} collapse construction.)
\end{enumerate}
If $W$ is a smooth closed $d$-manifold and $\str_W: \Fr(TW) \to \Theta$
is an equivariant map, we may compose~(\ref{eq:8}) and~(\ref{eq:9}) to
get a map of spectra $S^0 \to MT\Theta$, i.e.\ a point
in $\Omega^\infty MT\Theta$.  We shall write $\alpha(W,\str_W) \in \Omega^\infty MT\Theta$ for this point, and write
\begin{equation*}
  \Omega^\infty_{[W,\str_W]} MT\Theta \subset \Omega^\infty MT\Theta
\end{equation*}
for the path component containing $\alpha(W,\str_W)$.  This is the path component that the map~(\ref{eq:5}) lands in.

The map~(\ref{eq:5}) is given by a parametrised version of this
construction: given a family
$(\pi: E \to X, \str: \Fr(T_\pi E) \to \Theta)$ as in
Definition~\ref{def:family} over some base manifold $X$, it associates
a continuous map $\alpha: X \to \Omega^\infty MT\Theta$.  Said differently, it
comes from a composition of spectrum maps, the parametrised analogues
of~(\ref{eq:8}) and~(\ref{eq:9}) respectively,
\begin{equation}
  \label{eq:16}
  \begin{aligned}
    \Sigma^\infty_+ X & \lra E^{-T_\pi E}\\
    E^{-T_\pi E} &\lra MT\Theta.
  \end{aligned}
\end{equation}
In any case,~(\ref{eq:5}) is a universal version of this construction.

\begin{remark}
  Applying spectrum homology and the Thom isomorphism to the first map in~(\ref{eq:16}), we get precisely the fibre integration homomorphism~(\ref{eq:17}), while the second gives~(\ref{eq:20}).  This explains the connection to the characteristic classes in \S\ref{sec:char-class-1}.
\end{remark}

\begin{remark}\label{remark:disconnected}
  There is an improvement to Theorem~\ref{thm:homotopical}, in which the domain of
  (\ref{eq:5}) is replaced by a disconnected space
  $\mathcal{M}^\Theta_n$ containing $\mathcal{M}^\Theta(W,\str_W)$ as
  one of its path components, cf.\ \S\ref{sec:two-fiber-sequences} below and \cite[Section
  8]{GR-W4}.  In this improved formulation, the number
  $g(W,\str_W) \in \bN$ appearing in Theorem~\ref{thm:homotopical} is replaced by a function
  \begin{equation*}
    \pi_0(\mathcal{M}^\Theta_n) \overset{\alpha_*}\lra \pi_0(MT\Theta) \overset{\bar{g}^\Theta}\lra \bZ,
  \end{equation*}
  whose value on the path component $\mathcal{M}^\Theta(W,\str_W)$ is
  the \emph{stable genus}, defined in \cite[Section 5]{GR-W3} and
  \cite[Section 1.3]{GR-W4}.  It is at least $g(W,\str_W)$.

  If we write $\chi: \pi_0(MT\Theta) \to \Z$ and
  $\sigma: \pi_0(MT\Theta) \to \Z$ be the homomorphisms arising from
  the Euler class and (for even $n$) the Hirzebruch class,
  then~(\ref{eq:18}) defines a function
  $g^a: \pi_0(MT\Theta) \to \bN$.  In terms of this function, the
  estimate~(\ref{eq:15}) also holds for $\bar{g}^\Theta$.
\end{remark}

\subsection{General tangential structures}

The requirement in Theorems \ref{thm:main-cohomological} and \ref{thm:homotopical} that $\str_W: \Fr(TW) \to \Theta$ be $n$-connected
appears quite restrictive at first sight.  For example, it usually
rules out the interesting special cases $\Theta = \{\ast\}$ and
$\Theta = \{\pm 1\}$ from Remark~\ref{remark:special-cases}, so that
the cohomology of $B\Diff(W)$ and $B\Diff^\orientationpreserving(W)$ are not immediately calculated by Theorem~\ref{thm:homotopical}.

A more generally useful version of Theorem~\ref{thm:homotopical},
which holds without the connectivity assumption, may be deduced by a
rather formal homotopy theoretic trick, based on the observation that
the map~(\ref{eq:5}) is functorial in the $\GL_d(\R)$-space $\Theta$.

In particular, any map $\Theta \to \Theta$ induces a self-map of
$\Omega^\infty MT\Theta$.  The following result is \cite[Corollary
1.9]{GR-W4}.

\begin{theorem}\label{thm:general-str}
  For $d = 2n > 4$ and $\Lambda$ a $\GL_d(\R)$-space, let $W$ be a
  closed simply-connected smooth $d$-dimensional manifold, and
  $\lambda_W: \Fr(TW) \to \Lambda$ be an equivariant map. Choose an equivariant Moore--Postnikov\index{Moore--Postnikov stage} $n$-stage
\begin{equation}\label{eq:MPFactorisation}
\lambda_W : \Fr(TW) \overset{\str_W}\lra \Theta \overset{u}\lra \Lambda,
\end{equation}
i.e.\ a factorisation where $u$ is  $n$-co-connected equivariant fibration and $\str_W$ is an $n$-connected equivariant cofibration, and write $\hAut(u)$ for the group-like topological monoid consisting of equivariant weak equivalences $\Theta \to \Theta$ over $\Lambda$.

  This topological monoid acts on $\Omega^\infty MT\Theta$, and there
  is a continuous map
  \begin{equation}\label{eq:11}
    \alpha: \mathcal{M}^\Lambda(W,\lambda_M) \lra
    (\Omega^\infty MT\Theta)\moddd \hAut(u).
  \end{equation}
  which, when regarded as a map onto the path component that it hits,
  induces an isomorphism in homology with local coefficients in
  degrees $\leq (g(W,\lambda_W)-4)/3$, and if $\Theta$ is spherical it
  induces an isomorphism in integral homology in degrees
  $\leq (g(W,\lambda_W) - 3)/2$.
\end{theorem}

We emphasise that the homotopy orbit space is formed in the category
of spaces, not of infinite loop spaces (there is a comparison map
$(\Omega^\infty MT\Theta)\moddd \hAut(u) \to
\Omega^\infty((MT\Theta) \moddd \hAut(u))$ but it is not a weak
equivalence and we shall not need its codomain). Equivariant factorisations \eqref{eq:MPFactorisation} always exist, and are unique up to contractible choice.

Let us also point out that the group $\pi_0(\hAut(u))$ likely
acts non-trivially on $\pi_0(MT\Theta)$.  The construction of the
map~(\ref{eq:11}), outlined in \S\ref{sec:two-fiber-sequences}
below, together with the orbit-stabiliser theorem, lets us re-write
the relevant path component of the codomain of~(\ref{eq:11}) in the
following way, which is how the theorem above is typically used in
practice.

\begin{corollary}\label{cor:general-str}
  Let $d$, $\Lambda$, $W$, $\lambda_W$, $\Theta$, $\str_W$, and $u$ be as
  in Theorem~\ref{thm:general-str}, and write
  \begin{equation*}
    \hAut(u)_{[W,\str_W]} \subset \hAut(u)
  \end{equation*}
  for the submonoid stabilising the element
  $[W,\str_W] \in \pi_0(MT\Theta)$ defined in
  \S\ref{sec:mmm-class-gener}.  The action of this submonoid on
  $\Omega^\infty MT\Theta$ restricts to an action on the path
  component $\Omega^\infty_{[W,\str_W]} MT\Theta$ defined in
  \S\ref{sec:mmm-class-gener}, and~(\ref{eq:11}) factors through
  a map of path connected spaces
  \begin{equation*}
   \alpha:  \mathcal{M}^\Lambda(W,\lambda_W) \lra (\Omega^\infty_{[W,\str_W]} MT\Theta)\moddd(\hAut(u)_{[W,\str_W]})
  \end{equation*}
  which induces an isomorphism on homology in a range, as in
  Theorem~\ref{thm:general-str}. \qed
\end{corollary}

\begin{remark}
  In both Theorem~\ref{thm:general-str} and Corollary~\ref{cor:general-str} above, the range is expressed in terms of $g(W,\lambda_W)$ but as explained in \cite[Lemma 9.4]{GR-W4} this is equal to $g(W,\str_W)$ when $\lambda_W$ is factored equivariantly as an $n$-connected map $\str_W: \Fr(TW) \to \Theta$ followed by an $n$-co-connected map $u: \Theta \to \Lambda$.  Hence the estimates in \S\ref{sec:disc-moduli-spac} apply, when $e$ is the minimal number of generators for the abelian group $H_n(\Theta \moddd \GL_{2n}(\R))$.  The value of $e$ likely depends on the map $\lambda_W: \Fr(TW) \to \Lambda$, even if $W$ and $\Lambda$ are fixed.
\end{remark}

\begin{remark}\label{rem:finite-calculation}
  The fact that $u$ is $n$-co-connected implies that
  $\hAut(u)$ is an $(n-1)$-type.  Hence it is in some sense a finite problem to
  determine and describe $\hAut(u)$: finitely many homotopy
  groups $\pi_0, \dots, \pi_{n-1}$ and finitely many $k$-invariants.

  This finiteness is one of the conceptual advantages of our approach
  to $\cM(W) \simeq B\Diff(W)$, over the more classical method which
  at first gives a formula for the \emph{structure space}
  $\mathcal{S}(W) \simeq G(W) / \widetilde{\Diff}(W)$, where $G(W)$ is the monoid of homotopy equivalences from $W$ to itself and
  $\widetilde{\Diff}(W)$ is the block diffeomorphism group.\index{diffeomorphism group!block}  In that method, one subsequently
  has to study the difference between $\Diff(W)$\index{diffeomorphism group} and
  $\widetilde{\Diff(W)}$, but also has to take homotopy orbits by the
  monoid $G(W)$ of homotopy automorphisms of $W$.  While this last
  step is in some sense ``purely homotopy theory'', it is in practice
  very difficult to get a good handle on $G(W)$, even when $W$ is
  relatively simple and even when one is working up to rational
  equivalence.  See the work of Berglund and Madsen \cite{BerglundMadsen, BerglundMadsenII} for a recent example.  
\end{remark}

\subsection{Two fibre sequences}
\label{sec:two-fiber-sequences}

In \cite[Section 9]{GR-W4}, Theorem~\ref{thm:general-str} is deduced
from the special case given in Theorem~\ref{thm:homotopical} by a rather formal
argument: the homology equivalence in Theorem~\ref{thm:homotopical} is natural
in the $\GL_{2n}(\R)$-space $\Theta$, and hence induces a homology
equivalence by taking homotopy colimit over any diagram in
$\GL_{2n}(\R)$-spaces; in particular one may form homotopy orbits by
$\hAut(u: \Theta \to \Lambda)$, and
Theorem~\ref{thm:general-str} is deduced by identifying the resulting
map of homotopy orbit spaces with~(\ref{eq:11}).  The two fibre sequences
arising from these homotopy orbit constructions, see diagram~(\ref{eq:21}) below, are important for carrying out calculations in concrete examples, and hence we recall
this story in slightly more detail.

\begin{definition}
  Let $u: \Theta \to \Lambda$ be an equivariant $n$-co-connected
  fibration.
  \begin{enumerate}[(i)]
  \item Let $\mathcal{M}^\Theta_n \subset \mathcal{M}^\Theta$ be the
    union of those path components $\mathcal{M}^\Theta(W,\str_W)$ for
    which $\str_W: \Fr(TW) \to \Theta$ is $n$-connected.
  \item Let $\mathcal{M}^\Lambda_u \subset \mathcal{M}^\Lambda$ be the
    union of those path components $\mathcal{M}^\Lambda(W,\lambda_W)$ for which $\lambda_W$
    \emph{admits} a factorisation through an $n$-connected equivariant
    map $\str_W : \Fr(TW) \to \Theta$.
  \end{enumerate}
\end{definition}
There is an obvious forgetful map
$\mathcal{M}^\Theta_n \to \mathcal{M}^\Lambda_u$ given by composing
$\str_W : \Fr(TW) \to \Theta$ with $u$. The monoid $\hAut(u)$ has the correct homotopy type when $\Theta$ is equivariantly cofibrant, which we shall assume. It
acts on $\mathcal{M}^\Theta_n$ by postcomposing
$\str_W : \Fr(TW) \to \Theta$ with self-maps of $\Theta$.  This action
commutes with the forgetful map, and induces a map
\begin{equation}\label{eq:13}
  (\mathcal{M}^\Theta_n)\moddd \hAut(u) \lra
  \mathcal{M}^\Lambda_u.
\end{equation}
The following lemma is proved by elementary homotopy theoretic
methods, cf.\ \cite[Section 9]{GR-W4}.
\begin{lemma}
  The map~(\ref{eq:13}) is a weak equivalence.  Hence there is an
  induced fibre sequence of the form
  \begin{equation*}
    \mathcal{M}^\Theta_n \lra \mathcal{M}^\Lambda_u \lra B
    (\hAut(u)).
  \end{equation*}
\end{lemma}

The map $\mathcal{M}^\Theta_n \to \Omega^\infty MT\Theta$ explained in
\S\ref{sec:mmm-class-gener} commutes with the actions of
$\hAut(u)$ and induces a map of fibre sequences
\begin{equation}\label{eq:21}
\begin{tikzcd}
    \mathcal{M}^\Theta_n \rar \dar &\mathcal{M}^\Lambda_u
    \rar \dar & B(\hAut(u)) \arrow[equals]{d}\\
    \Omega^\infty MT\Theta \rar & (\Omega^\infty MT\Theta)\moddd
    \hAut(u) \rar & B(\hAut(u)).
\end{tikzcd}
\end{equation}
In the setup of Theorem~\ref{thm:general-str},
 $\mathcal{M}^\Theta(W,\str_W)$ is one path component of
$\mathcal{M}^\Theta_n$, and similarly $\mathcal{M}^\Lambda(W,\lambda_W)$
is one path component of $\mathcal{M}^\Lambda_u$.  A slightly stronger version of
Theorem~\ref{thm:homotopical}, which is the statement actually proved in \cite[Section 8]{GR-W4}, shows that the left-most vertical map is acyclic in the
range of degrees indicated in Remark~\ref{remark:disconnected}. Theorem~\ref{thm:general-str} is then deduced by a spectral sequence
comparison argument.

\begin{remark}\label{rem:Kreck}
This formulation has content even at the level of path components. Suppose that $W$ is a simply-connected $2n$-manifold for $2n \geq 6$, $\str_W : \Fr(TW) \to \Theta$ is $n$-connected, and $g(W, \str_W) \geq 3$. If $(W', \str_{W'})$ is another such $\Theta$-manifold and
$$\alpha(W, \str_W) = \alpha(W', \str_{W'}) \in \pi_0(\Omega^\infty MT\Theta) = \pi_0(MT\Theta)$$
then it follows that $(W, \str_W)$ and $(W', \str_{W'})$ lie in the same path-component of $\mathcal{M}^\Theta_n$, i.e.\ there is a diffeomorphism from $W$ to $W'$ which pulls back $\str_{W'}$ to $\str_W$ up to homotopy. This recovers a theorem of Kreck \cite[Theorem D]{Kreck}, though our requirement on genus is slightly stronger than Kreck's.
\end{remark}

In practice, one usually calculates the cohomology of $\Omega^\infty_{[W,\str_W]} MT\Theta$
first, and then uses a spectral sequence to calculate the homology or cohomology of the Borel construction in  Corollary~\ref{cor:general-str}, or equivalently (for $g(W, \lambda_W) \geq 3$) one calculates the cohomology of
$\mathcal{M}^\Theta(W,\str_W)$ and then uses the spectral sequence for the fibre sequence
\begin{equation}\label{eq:22}
\cM^\Theta(W, \str_W)\lra \cM^\Lambda(W, \lambda_W) \lra B(\hAut(u)_{[W, \str_W]}).
\end{equation}
We shall see examples of such calculations in \S\ref{sec:WgRat}, \S\ref{sec:VdRat}, and \S\ref{sec:OtherRat}.

\subsection{$\GL_d(\R)$-spaces versus spaces over $B\OO(d)$}
\label{sec:gl_dr-spaces-versus}

It is well known that the homotopy theory of spaces with action of $\GL_d(\R)$ is equivalent to the homotopy theory of spaces over $B\GL_d(\R) \simeq B\OO(d)$, where the weak equivalences are the equivariant maps that are weak equivalences of underlying spaces, respectively fibrewise maps that are weak equivalences of underlying spaces.  The translation goes via the space $E\GL_d(\R)$ which simultaneously comes with an action of $\GL_d(\R)$ and a map to $B\GL_d(\R)$.  Explicitly, given a $\GL_d(\R)$-space $\Theta$, the Borel construction $B = \Theta \moddd \GL_d(\R)$ comes with a map $B \to B\GL_d(\R)$; conversely, given a space $B$ and a map $\theta: B \to B\GL_d(\R)$ the fibre product $\Theta = E\GL_d(\R) \times_{B\GL_d(\R)} B$ comes with an action; these processes are inverse up to (equivariant/fibrewise) weak equivalence, as $E\GL_d(\R)$ is contractible.

Therefore all of the theorems above that depend on a $\GL_d(\R)$-space $\Theta$ may be stated in equivalent ways taking as input a space $B$ and a map $\theta: B \to B\OO(d)$.\index{$\theta$-structure} In the papers
\cite{GR-W2, GR-W3, GR-W4} we have taken the latter point of view.  
In this picture, a $\theta$-structure on a manifold $W$ is a (fibrewise linear) vector bundle map $\hat{\ell}_W : TW \to \theta^* \gamma$, where $\gamma$ denotes the universal vector bundle on $B\GL_d(\R)$.  As in those papers, we shall use the notation
\begin{equation*}
  MT\theta = MT\Theta,
\end{equation*}
when $\theta: B \to B\OO(d)$ is the map corresponding to the
$\GL_d(\R)$-space $\Theta$; i.e., $MT\theta = B^{-\theta}$ is the Thom
spectrum of the virtual inverse of the vector bundle classified by
$\theta: B \to B\OO(d)$.

We have already seen definitions which may be stated
more directly in terms of $(B,\theta)$ than of
$(\Theta,\text{action})$, e.g.\ the characteristic classes $\kappa_c$
from \S\ref{sec:char-class-1}. The constructions in
Theorem~\ref{thm:general-str} form another such example, which we shall now explain.  Given $\Lambda$ and $\lambda_W$ as in the theorem, set
$B' = \Lambda \moddd \GL_d(\R)$.  Up to contractible choice, $\lambda_W$
induces a map $W \to B'$, which one then Moore--Postnikov factors as
\begin{equation*}
  W \lra B \lra B',
\end{equation*}
into an $n$-connected cofibration followed by an $n$-co-connected
fibration.  In this picture, $\hAut(u)$ is simply the
group-like monoid of those self-maps of $B$ over $B'$ that are weak
equivalences. For this to have to the correct homotopy type $B$ should be fibrant and cofibrant in the category of spaces over $B'$.

In \S\ref{sec:RatCalc} and \S\ref{sec:AbCalc} we will exclusively adopt this point of view.

\subsection{Boundary}
\label{sec:boundary}

A further generalisation, also proved in \cite[Section 9]{GR-W4},
allows the compact manifolds $W$ to have non-empty boundary. The boundary should then be a closed
$(2n-1)$-manifold $P$, which should be equipped with an equivariant map
$\str_P: \Fr(\epsilon^1 \oplus TP) \to \Theta$.  The pair $(P,\str_P)$
should be fixed and every compact $2n$-manifold in sight should come
with a specified diffeomorphism $\partial W \cong P$ compatible with a
structure $\str_W: \Fr(TW) \to \Theta$.

In terms of classified bundles as in \S\ref{sec:smooth-bundles}
and \S\ref{sec:classifying-spaces}, $\class^\Theta_\bullet$ should be
replaced with the functor $\class_\bullet^{\Theta, P,\str_P}$ whose value on a smooth manifold
$X$ (without boundary, possibly non-compact) has 0-simplices the
smooth proper maps $\pi: E \to X$ equipped with equivariant maps
$\str: \Fr(T_\pi E) \to \Theta$ and a diffeomorphism $\partial E = X \times P$
such that the restriction of $\str$ to
$\Fr(T_\pi E \vert_{\partial E}) = X \times \Fr(TP)$ is equal to the map arising
from $\str_P$.

This kind of bundle also admits a classifying space, denoted
$\mathcal{N}^\Theta(P,\str_P)$ and called the \emph{moduli space of
  null bordisms}\index{moduli space!of null bordisms} of $(P,\str_P)$.  The subspace defined by the
condition that $\str_W: \Fr(TW) \to \Theta$ be $n$-connected is denoted
$\mathcal{N}^\Theta_n(P,\str_P)$ and is the \emph{moduli space of
  highly connected null bordisms}.\index{moduli space!of highly-connected null bordisms}  The path component of
$\mathcal{N}^\Theta(P,\str_P)$ containing $(W,\str_W)$ shall be
denoted $\mathcal{M}^\Theta(W,\str_W)$, as before.  Notice that
$(P,\str_P)$ is determined by $P = \partial W$ and $\str_W$ by
restricting $\str_W$ to $TW\vert_P \cong \epsilon^1 \oplus TP$.  These classifying spaces are introduced in \cite[Definition 1.1]{GR-W4} using a similar notation.

Theorem~\ref{thm:homotopical} then has the following direct generalisation, also stated as \cite[Corollary 1.8 and Section 8.4]{GR-W4}.

\begin{theorem}
  Let $d = 2n > 4$, let $\Theta$ be a $\GL_d(\R)$-space,
  let $P$ be a closed smooth $(d-1)$-manifold and $\str_P: \Fr(\epsilon^1 \oplus TP) \to \Theta$ be a $\GL_d(\R)$-equivariant map. Then there is a map (canonical up to homotopy, see below)
  \begin{equation}\label{eq:14}
    \alpha: \mathcal{N}^\Theta_n(P,\str_P) \lra
    \Omega_{\alpha(P,\str_P),0}(\Omega^{\infty-1}MT\Theta),
  \end{equation}
  where $\Omega_{\alpha(P,\str_P),0}$ denotes the space of paths
  starting at a certain point
  $\alpha(P,\str_P) \in \Omega^{\infty-1}MT\Theta$ and ending at the
  basepoint, with the following property.

  When restricted to the path component containing a
  particular $(W,\str_W)$, it is a homology equivalence onto the path
  component it hits, in degrees up to $(g(W,\str_W) - 4)/3$ and
  possibly with twisted coefficients.  If in addition $\Theta$ is
  spherical, then~(\ref{eq:14}) induces an isomorphism in homology
  with constant coefficients in degrees up to $(g(W,\str_W) - 3)/2$.
\end{theorem}

Both the point $\alpha(P,\str_P)$ and the map~(\ref{eq:14}) are
constructed by the procedure in \S\ref{sec:mmm-class-gener}.  If
the codomain of~(\ref{eq:14}) is non-empty, it is of course non-canonically homotopy
equivalent to $\Omega^\infty MT\Theta$, but the map most naturally
takes values in the path space.  In the special case $P = \emptyset$ we have $\mathcal{N}^\Theta_n(\emptyset) = \mathcal{M}^\Theta_n$, and the map \eqref{eq:14} is the same as the map appearing in~(\ref{eq:21}).

As in \S\ref{sec:two-fiber-sequences}, a version for a general tangential structure $\Lambda$ may be deduced by taking homotopy orbits with respect to the monoid
\begin{equation}\label{eq:19}
  \hAut(u \ \mathrm{rel}\  P) = \{\phi \in \hAut(u) \mid
  \phi \circ \str_P = \str_P\},
\end{equation}
provided $\str_P: \Fr(\epsilon^1 \oplus TP) \to \Theta$ is an
equivariant cofibration and $u: \Theta \to \Lambda$ is an equivariant
$n$-co-connected fibration, as can be arranged.  Homotopy equivalently, factor the induced map $W \to B' = \Lambda \moddd \GL_d(\R)$ as an $n$-connected cofibration
$W \to B$ followed by an $n$-co-connected fibration $B \to B'$, and
define $\hAut(u \ \mathrm{rel}\ P)$ as the homotopy
equivalences of $B$ over $B'$ and under $P$.  We formulate the conclusion.
\begin{theorem}
  \label{thm:main-with-boundary}
  Let $n$, $d$, $\Lambda$ and $\lambda_W: \Fr(TW) \to \Lambda$ be as in Theorem~\ref{thm:general-str}, but allow $W$ to be a compact manifold with boundary $P = \partial W$.  Let $\Theta$, $\str_W$, and $u$ be as in Theorem~\ref{thm:general-str}, and $\str_P$ denote the restriction of $\str_W$ to $P$.  Then there is a map
  \begin{equation}
    \alpha: \mathcal{M}^\Lambda(W,\lambda_W) \lra \big(\Omega_{\alpha(P,\str_P),0} \Omega^{\infty - 1} MT\Theta\big) \moddd \hAut(u \ \mathrm{rel}\ P)
  \end{equation}
  which, when regarded as a map onto the path component that it hits,
  induces an isomorphism in homology in a range of degrees, exactly as
  in Theorem~\ref{thm:general-str}.
\end{theorem}

This is \cite[Theorem 9.5]{GR-W4}, and the three lines following its proof.  The relevant path component may again be re-written using the orbit-stabiliser theorem, as in Corollary~\ref{cor:general-str}.

\begin{remark} 
As in Remark \ref{remark:special-cases}, in the special case $\Theta=\{*\}$ a $\Theta$-structure contains no information and we can simply write $\mathcal{M}(W)$ for $\mathcal{M}^\Theta(W,\str_W)$. This space classifies smooth fibre bundles with fibres diffeomorphic to $W$ and trivialised boundary, and we have a weak equivalence
$$\mathcal{M}(W) \simeq B\Diff_\partial(W)$$
where $\Diff_\partial(W)$ denotes the group of diffeomorphisms of $W$ which fix an open neighbourhood of the boundary, with the $C^\infty$ topology.
\end{remark}

There are no connectivity assumptions imposed on
$\str_P: \Fr(\epsilon^1 \oplus TP) \to \Theta$, but if it happens to
be $(n-1)$-connected then the monoid $\hAut(u \ \mathrm{rel}\  P)$
is contractible.  More generally we have the following.

\begin{lemma}\label{lem:aHutContr}
  If the pair $(W, P)$ is $c$-connected for some $c \leq n-1$, then the monoid $\hAut(u \ \mathrm{rel}\  P)$ is a (non-empty) $(n-c-2)$-type.  In particular, it is contractible if $(W,P)$ is $(n-1)$-connected.
\end{lemma}
A familiar special case of this observation is the fact that if a
diffeomorphism of an oriented surface with non-empty boundary is the
identity on the boundary, then the diffeomorphism is automatically
orientation preserving.
\begin{proof}
  As before, let us write $B = \Theta \moddd \GL_d(\R)$ and
  $B' = \Lambda \moddd \GL_d(\R)$.  We are then asking for homotopy
  automorphisms of $B$ over $u: B \to B'$ and under $\ell_P :P \hookrightarrow B$. By adjunction, to give a nullhomotopy of a map
  $f : S^k \to \hAut(u \ \mathrm{rel}\ P)$ is to solve the relative
  lifting problem
  \begin{equation*}
    \begin{tikzcd} 
      (S^k \times B) \cup_{S^k \times P} (D^{k+1} \times P) \dar \arrow{rr}{\tilde{f} \cup (\ell_{P} \circ \text{proj})} &  & B \dar{u}\\
      D^{k+1} \times B \rar{\text{proj}} & B \rar{u} & B'.
    \end{tikzcd}
  \end{equation*}
  The map $\ell_P: P \to B$ is $c$-connected because both $P \subset W$ and $\ell_W: W \to B$ are (the latter is even $n$-connected). Thus the pair
  \begin{equation*}
    (D^{k+1} \times B, (S^k \times B) \cup_{S^k \times P} (D^{k+1} \times P))
  \end{equation*}
  is $(c+k+1)$-connected. But the map $u: B \to B'$ is
  $n$-co-connected, so there are no obstructions to solving this
  lifting problem if $c+k+1 \geq n$, i.e.\ $k \geq n - c - 1$.  This proves that $\hAut(u\ \mathrm{rel}\ P)$ is an $(n-c-2)$-type, and it is non-empty because it contains the identity map.
\end{proof}

\begin{remark}
  Formulating a statement which is valid for manifolds with non-empty boundary is not purely for the purpose of added generality: it is essential for the strategy of proof in all three papers \cite{GR-W2}, \cite{GR-W3}, \cite{GR-W4}.  For example, the homological stability results in \cite{GR-W3} are proved by a long \emph{handle induction} argument, in which a compact manifold is decomposed into finitely many \emph{handle attachments}; even if one is mainly interested in closed manifolds, this process will create boundary.  Similarly, an important role in both \cite{GR-W2} and \cite{GR-W4} is played by \emph{cobordism categories} as studied in \cite{GMTW}, whose morphisms are manifolds with boundary and composition is gluing along common boundary components. For example, given a $\Theta$-cobordism $(K, \str_K) : (P, \str_P) \leadsto (Q, \str_Q)$ there is a continuous map
$$(K, \str_K) \cup - : \mathcal{N}^\Theta(P,\str_P) \lra \mathcal{N}^\Theta(Q,\str_Q)$$
given by gluing on $(K, \str_K)$.
\end{remark}

\subsection{Fundamental group}\label{sec:pi1}

The main theorem in either of the three forms given above
(Theorems~\ref{thm:main-cohomological}, \ref{thm:homotopical},
and~\ref{thm:general-str}) assumed the manifolds $W$ were simply-connected, but in fact it suffices that the fundamental groups
$\pi = \pi_1(W,w)$ be \emph{virtually polycyclic},\index{virtually polycyclic} i.e.\ has a subnormal series with finite or cyclic quotients. In this case the \emph{Hirsch length} $h(\pi)$ is the number of infinite cyclic quotients in such a series. The only price to pay is that the ranges of homology equivalence
become offset by a constant depending on $h(\pi)$: the homology
isomorphisms Theorems~\ref{thm:homotopical} and~\ref{thm:general-str}
hold in degrees $\leq (g(W,\str_W) - (h(\pi) + 5))/2$ with integral
coefficients if $\Theta$ is spherical, and in degrees $\leq (g(W,\str_W) - (h(\pi) + 6))/3$ with
local coefficients. This generalisation was established by
Friedrich \cite{NinaPaper}.

For arbitrary $\pi$ there is a sense in which the theorems
hold in ``infinite genus'': certain maps become acyclic after taking a
colimit over forming connected sum with $S^n \times S^n$ infinitely
many times. In this form the assumption $2n > 4$ is also unnecessary.
See \cite[Sections 1.2 and 7]{GR-W4} for the statement and proof.

\subsection{Outlook}
\label{sec:outlook}

We have attempted to give an overview of the methods developed in
\cite{GR-W2}, \cite{GR-W3}, \cite{GR-W4}, with an emphasis on the main
results from there as they may be applied in calculations in practice.
This is by no means a survey of everything known, let us briefly
mention some recent developments and applications that we have not covered:

\begin{enumerate}[(i)]
\item These results---in the form of the calculation described in Theorem \ref{thm:Wg1Rat} below---have been used by Weiss \cite{WeissDalian} to prove that $p_n \neq e^2 \in H^{4n}(B\mathrm{STop}(2n);\bQ)$ for large enough $n$. These methods were later used by Kupers \cite{KupersFin} to establish the finite generation of homotopy groups of $\Diff_\partial(D^d)$ for $d \neq 4,5,7$.
\item These results have been used by Botvinnik, Ebert, and Randal-Williams \cite{BER-W}, Ebert and Randal-Williams \cite{ER-Wpsc}, and Botvinnik, Ebert, and Wraith \cite{BEW} to study the topology of spaces of Riemannian metrics of positive scalar, or Ricci, curvature.
\item These results have been used by Krannich \cite{KrannichExotic} to show that if $\Sigma$ is a homotopy $2n$-sphere\index{homotopy sphere} then $B\Diff(M)$ and $B\Diff(M \# \Sigma)$ have the same homology in the stable range with $\bZ[\frac{1}{k}]$-coefficients, where $k$ is the order of $\Sigma$ is the group of homotopy spheres.
\item Progress towards a similar understanding for manifolds of \emph{odd dimension} has been made by Perlmutter \cite{PerlmutterProdSpheres, PerlmutterLink, PerlmutterStabHand}, Botvinnik and Perlmutter \cite{BotvinnikPerlmutter} and Hebestreit and Perlmutter \cite{HebestreitPerlmutter}.
\item Progress towards versions for \emph{topological} and \emph{piecewise linear} manifolds has been made by Gomez-Lopez \cite{LopezThesis}, Kupers \cite{KupersTopStab}, and Gomez-Lopez and Kupers \cite{LopezKupers}.
\item Progress towards versions for \emph{equivariant smooth} manifolds has been made by Galatius and Sz\H{u}cs \cite{SzucsGalatius}.
\item There are analogues of many of the theorems above, when the topological group $\Diff(W)$ is replaced by its \emph{underlying discrete group}, by Nariman \cite{Nariman1, Nariman2, Nariman3}.
\item Progress towards understanding the homotopy equivalences of high genus manifolds have been obtained by Berglund and Madsen \cite{BerglundMadsenII}.  At present their results seem to be qualitatively quite different from the results described here.
\end{enumerate}

\section{Rational cohomology calculations}\label{sec:RatCalc}

We have already advertised the feature that the general theory surveyed in \S\ref{sec:char-class-XXX} and \S\ref{sec:gener-vers-main} above is amenable to explicit calculations.  In this section and the next, we back up this claim with some examples, while simultaneously illustrating how the abstract homotopy theory in \S\ref{sec:gener-vers-main} plays out in concrete examples.

In practice, given a manifold $W$ and a structure $\lambda_W: \mathrm{Fr}(W) \to \Lambda$, one typically first estimates the genus of $(W,\lambda_W)$.  This step is trivial for the manifolds considered in \S\ref{sec:RatWgs} and \S\ref{sec:WgRat} below, which are defined to have high genus, but is quite interesting for the complete intersections considered in \S\ref{sec:VdRat}.  The next step would typically be to determine the associated highly-connected structure $\str_W: \mathrm{Fr}(W) \to \Theta$ and the space $B\hAut(u\ \mathrm{rel}\ P)$.  This step is mostly resolved by Lemma~\ref{lem:aHutContr} for the example given in \S\ref{sec:RatWgs}, but is again interesting for the examples in \S\ref{sec:WgRat} and especially \S\ref{sec:VdRat}.  For calculations in rational cohomology, the last step would then typically be to understand the Serre spectral sequence associated to~(\ref{eq:22}).  The complete intersections in \S\ref{sec:VdRat} again provide an interesting and very non-trivial illustration.

\subsection{The manifolds $W_{g,1}$}\label{sec:RatWgs}

Recall that we write $W_g = g(S^n \times S^n)$ for the $g$-fold connected sum, and $W_{g,1} = D^{2n} \# W_g \cong W_g \setminus \mathrm{int}(D^{2n})$. These manifolds play a distinguished role in the theory described above, as they are used to measure the genus of arbitrary $2n$-manifolds: in this sense $W_{g,1}$ is the simplest manifold of genus $g$. In the case $2n=2$ the solution by Madsen and Weiss \cite{MW} of the Mumford Conjecture\index{Mumford Conjecture}\index{Madsen--Weiss theorem} gave a description of $H^*(\cM^\orientationpreserving(W_{g,1});\bQ)$ in terms of Miller--Morita--Mumford classes in a stable range of degrees. In this section we wish to explain how the analogue of Madsen and Weiss' result in dimensions $2n \geq 6$ follows from the theory described above.

The Moore--Postnikov $n$-stage 
$$\tau_{W_{g,1}} : W_{g,1} \overset{\ell_{W_{g,1}}}\lra B \overset{\theta}\lra B\OO(2n)$$
of a map classifying the tangent bundle of $W_{g,1}$ has the cofibration $\ell_{W_{g,1}}$ $n$-connected and the fibration $\theta$ $n$-co-connected. As $W_{g,1}$ is $(n-1)$-connected and the map $\tau_{W_{g,1}}$ is nullhomotopic---because $W_{g,1}$ admits a framing---we may identify $\theta$ with the $n$-connected cover of $B\OO(2n)$, which we write as
$$\theta_n : B\OO(2n)\langle n\rangle \overset{\theta_n^\orientationpreserving}\lra B\SO(2n) \overset{\orientationpreserving}\lra B\OO(2n).$$

As the pair $(W_{g,1}, \partial W_{g,1})$ is $(n-1)$-connected, by Lemma \ref{lem:aHutContr} we find that $\hAut(\theta_n^\orientationpreserving, \ell_{\partial W_{g,1}})$ is contractible. Thus by Theorem \ref{thm:main-with-boundary} there is a map
$$\alpha : \cM^\orientationpreserving(W_{g,1}) \simeq \cM^{\theta_n}(W_{g,1},\ell_{W_{g,1}}) \lra \Omega^\infty MT\theta_n$$
which is a homology equivalence onto the path component that it hits, in degrees $* \leq \frac{g-3}{2}$, as long as $2n \geq 6$.
The rational cohomology of a path component of $\Omega^\infty MT\theta_n$ is calculated as described in Remark~\ref{rem:RatCalc}, in terms of $H^*(B\OO(2n)\langle n \rangle ;\bQ)$. To work this out we identify $B\OO(2n)\langle n\rangle = B\SO(2n)\langle n\rangle$. As
$$H^*(B\SO(2n);\bQ) = \bQ[e, p_1, p_2, \ldots, p_{n-1}]$$
is a free graded-commutative algebra the effect of taking the $n$-connected cover on cohomology groups is a simple as possible: it simply eliminates all free generators of degree $ \leq n$. Thus
$$H^*(B\SO(2n)\langle n \rangle ;\bQ) = \bQ[e, p_{\lceil \frac{n+1}{4} \rceil}, \ldots, p_{n-1}].$$

Combining all the above, we obtain the following.

\begin{theorem}\label{thm:Wg1Rat}
For $2n \geq 6$, let $\mathcal{B}$ denote the set of monomials in the classes $e$, $p_{n-1}$, $p_{n-2}$, \ldots, $p_{\lceil \frac{n+1}{4}\rceil}$. Then the map
$$\bQ[\kappa_c \,\,\vert\,\, c \in \mathcal{B}, |c| > 2n] \lra H^*(\cM^\orientationpreserving(W_{g,1});\bQ)$$
is an isomorphism in degrees $* \leq \frac{g-3}{2}$.
\end{theorem}
This is \cite[Corollary 1.8]{GR-W3}. As we mentioned above, for $2n=2$ the same statement (with a slightly different range) was earlier proved by Madsen and Weiss.

\subsection{The manifolds $W_g$}\label{sec:WgRat}
For $2n=2$ it is a theorem of Harer \cite{H} that the map
\begin{equation}\label{eq:CloseBdy}
\cM^\orientationpreserving(W_{g,1}) \lra \cM^\orientationpreserving(W_{g})
\end{equation}
given by attaching a disk induces an isomorphism on homology in a stable range of degrees. On the other hand $\cM^\orientationpreserving(W_{g})$ is the homotopy type of the stack $\mathbf{M}_g$ of genus $g$ Riemann surfaces, and it is in this way that Madsen and Weiss' topological result determines $H^*(\mathbf{M}_g;\bQ)$ in a stable range.

In dimensions $2n \geq 6$ it is no longer true that the map \eqref{eq:CloseBdy} induces an isomorphism on homology in a stable range of degrees. In this section we shall explain why, by using the theory described above to calculate $H^*(\cM^\orientationpreserving(W_{g});\bQ)$ in a stable range.

Continuing to write 
$$\theta_n : B\OO(2n)\langle n\rangle \overset{\theta_n^\orientationpreserving}\lra B\SO(2n) \overset{\orientationpreserving}\lra B\OO(2n)$$
as in the previous section, this tangential structure is also the Moore--Postnikov $n$-stage of the map classifying the tangent bundle of $W_g$. By Theorem \ref{thm:general-str} there is a map
$$\alpha : \cM^\orientationpreserving(W_g) \lra (\Omega^\infty MT\theta_n) \hcoker \hAut(\theta^\orientationpreserving_n)$$
that induces an isomorphism on homology, onto the path component that it hits, in degrees $* \leq \tfrac{g-3}{2}$ as long as $2n \geq 6$. 

In order to calculate the rational cohomology of $\cM^\orientationpreserving(W_g)$ in a range of degrees we could try to calculate the rational cohomology of the relevant path component of $(\Omega^\infty MT\theta_n) \hcoker \hAut(\theta^\orientationpreserving_n)$, but instead we shall use the fibre sequence \eqref{eq:22}. Choosing a $\theta_n$-structure $\ell_{W_g}$ on $W_g$, this is a fibre sequence
\begin{equation}\label{eq:WgFibSeq}
\cM^{\theta_n}(W_g, \ell_{W_g}) \lra \cM^\orientationpreserving(W_g) \overset{\xi}\lra B(\hAut(\theta_n^\orientationpreserving)_{[W_g, \ell_{W_g}]}).
\end{equation}
As $\ell_{W_g} : W_g \to B\OO(2n)\langle n\rangle$ is $n$-connected, we may apply Theorem \ref{thm:main-cohomological}, giving that
$$\bQ[\kappa_c \,\,\vert\,\, c \in \mathcal{B}, |c| > 2n] \lra H^*(\cM^{\theta_n}(W_{g},\ell_{W_g});\bQ)$$
is an isomorphism in degrees $* \leq \frac{g-3}{2}$. All the classes $\kappa_c$ are defined on the total space $\cM^\orientationpreserving(W_g)$ of the fibration \eqref{eq:WgFibSeq}, which implies that this fibration satisfies the Leray--Hirsch property in the stable range. Thus the Leray--Hirsch map
\begin{equation}\label{eq:27}
\bQ[\kappa_c \,\,\vert\,\, c \in \mathcal{B}, |c| > 2n] \otimes H^*(B(\hAut(\theta_n^\orientationpreserving)_{[W_g, \ell_{W_g}]});\bQ) \lra H^*(\cM^\orientationpreserving(W_{g});\bQ)
\end{equation}
is an isomorphism in degrees $* \leq \frac{g-3}{2}$.

To complete this calculation we must calculate the rational cohomology of the space $B(\hAut(\theta_n^\orientationpreserving)_{[W_g, \ell_{W_g}]}$, and describe the map
$$\xi^* : H^*(B(\hAut(\theta_n^\orientationpreserving)_{[W_g, \ell_{W_g}]});\bQ) \lra H^*( \cM^\orientationpreserving(W_g);\bQ)$$
in terms that we understand.

\subsubsection{Identifying $\hAut(\theta^\orientationpreserving_n)$}

The map $\theta^\orientationpreserving_n : B\OO(2n)\langle n \rangle \to B\SO(2n)$ is a principal fibration for the group-like topological monoid $\SO[0,n-1]$, the truncation of $\SO$, so is classified by the map $B\SO(2n) \to B\SO \to B(\SO[0,n-1])$, and hence there is a homomorphism
$$\iota : \SO[0,n-1] \lra \hAut(\theta^\orientationpreserving_n)$$
given by the principal group action.

\begin{lemma}\label{lem:WgIdhAut}
The map $\iota$ is a weak homotopy equivalence.
\end{lemma}
\begin{proof}
If we fix a basepoint $* \in B\OO(2n)\langle n \rangle$, which identifies the fibre through $*$ with $\SO[0,n-1]$, then there is a map
\begin{align*}
ev: \hAut(\theta^\orientationpreserving_n) &\lra \SO[0,n-1]\\
\varphi & \longmapsto \varphi(*),
\end{align*}
and it is clear that $ev \circ \iota$ is the identity. Thus it is enough to show that $ev$ is a weak homotopy equivalence. Suppose we are given a map $(f, g): (D^{k+1}, S^k) \to (\SO[0,n-1], \hAut(\theta_n^\orientationpreserving))$. This determines a relative lifting problem
\begin{equation*}
\begin{tikzcd} 
(S^k \times B\OO(2n)\langle n \rangle) \cup_{S^k \times \{*\}} (D^{k+1} \times \{*\}) \rar{g \cup f} \dar&  B\OO(2n)\langle n \rangle \dar{\theta_n^\orientationpreserving}\\
D^{k+1} \times B\OO(2n)\langle n \rangle \rar{\theta_n^\orientationpreserving \circ \pi_2} & B\SO(2n)
\end{tikzcd} 
\end{equation*}
and finding a nullhomotopy of $(f, g)$ is the same as solving this relative lifting problem. The obstructions for doing so lie in the groups
$$\widetilde{H}^{i-k}(B\OO(2n)\langle n \rangle; \pi_{i+1}(B\SO(2n), B\OO(2n)\langle n \rangle)),$$
but these groups are zero if $i-k \leq n$ or if $i \geq n$, so always vanish.
\end{proof}

In particular, the submonoid $\hAut(\theta_n^\orientationpreserving)_{[W_g, \ell_{W_g}]} \leq \hAut(\theta_n^\orientationpreserving)$ is in fact the whole of $\hAut(\theta^\orientationpreserving_n)$, so we may identify the cohomology of its classifying space with
\begin{equation}\label{eq:29}
H^*(B\hAut(\theta^\orientationpreserving_n);\bQ) = H^*(B\SO[0,n];\bQ) = \bQ[p_1, p_2, \ldots, p_{\lfloor \frac{n}{4}\rfloor}].
\end{equation}

\subsubsection{Miller--Morita--Mumford class interpretation}

Combining~(\ref{eq:27}) and~(\ref{eq:29}) gives a formula for $H^*(\cM^\orientationpreserving(W_g);\bQ)$ in a range of degrees.  In fact the classes obtained by pulling back $p_1, p_2, \ldots, p_{\lfloor \frac{n}{4}\rfloor}$ along the map
$$\xi  : \cM^\orientationpreserving(W_g) \lra  B\hAut(\theta^\orientationpreserving_n)$$
may be re-interpreted as Miller--Morita--Mumford classes.  We shall use the following lemma to explain this.

\begin{lemma}\label{lem:VTangDescends}
Let $\pi^\orientationpreserving: \cE^\orientationpreserving(W_g) \to \cM^\orientationpreserving(W_g)$ denote the the path component of the fibration \eqref{eq:25} modelling the universal oriented $W_g$-bundle, and $\tau : \cE^\orientationpreserving(W_g) \to B\SO(2n)$ denote the map classifying the vertical tangent bundle. Then the square
\begin{equation}\label{eq:VTangDescends}
\begin{gathered}
\begin{tikzcd} 
\cE^\orientationpreserving(W_g) \arrow[rr,"\tau"] \dar{\pi^\orientationpreserving} & & B\SO(2n) \dar \\
\cM^\orientationpreserving(W_g) \rar{\xi} & B\SO[0,n] & B\SO(2n)[0,n] \lar[']{\simeq}
\end{tikzcd}
\end{gathered}
\end{equation}
commutes up to homotopy.
\end{lemma}
\begin{proof}
  Let $\pi^{\theta_n} : \cE^{\theta_n}(W_g,\ell_{W_g}) \to \cM^{\theta_n}(W_g, \ell_{W_g})$ denote the path component of the fibration modelling the universal $W_g$-bundle with $\theta_n$-structure, which as in \eqref{eq:25}  comes with maps
$$\tau : \cE^{\theta_n}(W_g,\ell_{W_g}) \overset{\ell}\lra B\OO(2n)\langle n \rangle \overset{\theta_n^\orientationpreserving}\lra B\SO(2n)$$
whose composition classifies the (oriented) vertical tangent bundle. This gives a commutative square
\[\begin{tikzcd} 
\cE^{\theta_n}(W_g,\ell_{W_g}) \dar{\pi^{\theta_n}} \rar{\ell}& B\OO(2n)\langle n \rangle \dar\\
\cM^{\theta_n}(W_g, \ell_{W_g}) \rar& *
\end{tikzcd}\]
of $\hAut(\theta^\orientationpreserving_n)$-spaces and $\hAut(\theta^\orientationpreserving_n)$-equivariant maps. Taking homotopy orbits, and replacing the spaces at each corner with homotopy equivalent models, we obtain the homotopy commutative square
\[\begin{tikzcd} 
\cE^\orientationpreserving(W_g) \dar{\pi^\orientationpreserving} \rar{\tau}& B\SO(2n) \dar\\
\cM^\orientationpreserving(W_g) \rar{\xi}& B\SO[0,n].
\end{tikzcd}\]
Here the right-hand map is the truncation $B\SO(2n) \to B\SO(2n)[0,n]$ followed by the identification $B\SO(2n)[0,n] \overset{\sim}\to B\SO[0,n]$, as required.
\end{proof}

Now we may calculate as follows: by this lemma we have
$$(\pi^\orientationpreserving)^* \xi^*(p_i) = \tau^*(p_i),$$
and so by the projection formula (and commutativity of the cup product)
$$\kappa_{ep_i} = \int_{\pi^\orientationpreserving} \tau^*(e \cdot p_i) = \left(\int_{\pi^\orientationpreserving} \tau^*e \right) \cdot \xi^*(p_i) = \chi(W_g) \cdot \xi^*(p_i)$$
and hence, for $\chi(W_g) = 2 + (-1)^n 2g \neq 0$, we have $\xi^*(p_i) = \frac{1}{\chi(W_g)} \kappa_{ep_i}$.

Combined with the previous discussion, we obtain the following.

\begin{theorem}\label{thm:WgRat}
For $2n \geq 6$, let $\mathcal{B}$ denote the set of monomials in the classes $e$, $p_{n-1}$, $p_{n-2}$, \ldots, $p_{\lceil \frac{n+1}{4}\rceil}$, and $\mathcal{C}$ denote the set of the remaining Pontryagin classes $p_1, p_2, \ldots, p_{\lfloor \frac{n}{4}\rfloor}$. Then the map
$$\bQ[\kappa_c \,\,\vert\,\, c \in (\mathcal{B} \sqcup e\cdot\mathcal{C}), |c|>2n] \lra H^*(\cM^\orientationpreserving(W_g);\bQ)$$
is an isomorphism in degrees $* \leq (g-3)/2$.
\end{theorem}

For $2n=2$ the same statement (with a slightly different range) holds by the theorem of Madsen and Weiss, and in this case the set $\mathcal{C}$ is empty and the result is the same as that of Theorem \ref{thm:Wg1Rat}. For $2n \geq 6$ we have $\hAut(\theta^\orientationpreserving_n) \simeq \SO[0,n-1] \not\simeq *$ and so it follows from the discussion in this section that the map \eqref{eq:CloseBdy} is \emph{not} an isomorphism on integral cohomology in any range of degrees (though for $2n=6$ it is still an isomorphism on rational cohomology in a stable range, as $\SO[0,2] \simeq K(\bZ/2,1)$ is rationally acyclic; in Theorem \ref{thm:WgRat} this corresponds to the fact that $\mathcal{C}$ is empty in this case).

\begin{remark}
In the statement of Theorem \ref{thm:WgRat} we do \emph{not} assert that $\kappa_c=0$ for monomials $c \nin \mathcal{B} \sqcup e\cdot\mathcal{C}$. Indeed a further consequence of the homotopy commutativity of \eqref{eq:VTangDescends} is the following description of $\kappa_c$ for a general monomial $c=e^i\cdot p_1^{j_1} \cdots p_{n-1}^{j_{n-1}}$ in terms of the generators of Theorem \ref{thm:WgRat}:
$$\kappa_c = \left(\frac{\kappa_{e\cdot p_1}}{\chi(W_g)}\right)^{j_1} \cdot \left(\frac{\kappa_{e\cdot p_2}}{\chi(W_g)}\right)^{j_2} \cdots \left(\frac{\kappa_{e\cdot p_k}}{\chi(W_g)}\right)^{j_k} \cdot \kappa_{\left(e^i \cdot p_{k+1}^{k_{k+1}} \cdots p_{n-1}^{j_{n-1}}\right)},$$
where we write $k = {\lfloor \frac{n}{4}\rfloor}$. This follows immediately from the observation that $p_i(T_\pi) = \pi^*\xi^*(p_i) = \pi^*(\frac{\kappa_{e\cdot p_i}}{\chi(W_g)})$ for $i \leq k$.
\end{remark}

\subsection{Hypersurfaces in $\bC\bP^4$}\label{sec:VdRat}

If $V \subset \bC\bP^{r+1}$ is a smooth hypersurface,\index{hypersurface} determined by a homogeneous complex polynomial of degree $d$, then it is an observation of Thom that its diffeomorphism type depends only on the degree $d$, and not on the particular polynomial: we call the resulting $2r$-manifold $V_{d}$.  As we shall explain in \S\ref{sec:genus-v_d}, these interesting manifolds tend to have large genus.  More generally, for smooth complete intersections\index{complete intersection} of such hypersurfaces the diffeomorphism type depends only on the degrees, and much is understood about the classification up to diffeomorphism of such manifolds in terms of these degrees, by Libgober and Wood \cite{LibgoberWood}, Kreck \cite{Kreck}, and others.

We shall explain how the theory described above applies in the non-trivial example of a hypersurface $V_d \subset \bC\bP^4$ of degree $d$, and determine a formula for the rational cohomology of $\cM^\orientationpreserving(V_d)$ in a range of degrees.  Let us start with outlining the steps again, and state the conclusions in this example.
\begin{enumerate}[(i)]
\item Determine the genus of $V_d$: it turns out to be $\frac{1}{2} (d^4-5d^3+10d^2-10d+4)$.
\item Determine the Moore--Postnikov 3-stage $V_d \xrightarrow{\ell_{V_d}} B_d \xrightarrow{\theta_d} B\SO(6)$ of a map classifying the oriented tangent bundle of $V_d$.
\item Calculate the ring $H^*(\cM^{\theta_d}(V_d,\ell_{V_d});\bQ)$ in the stable range. It turns out to be the $\bQ$-algebra
  \begin{equation}\label{eq:45}
    A = \bQ[\kappa_{t^n c} \mid \text{$c \in \mathcal{B}$, $n \geq 0$, $|c| + 2n > 6$}],
  \end{equation}
  where $\mathcal{B}$ is the set of monomials in classes $p_1$, $p_2$, and $e$ of degree $|p_1| = 4$, $|p_2| = 8$, and $|e| = 6$, and $t$ is a class of degree $2$.
\item Use the spectral sequence arising from Corollary~\ref{cor:general-str} to determine the cohomology of $\cM^\orientationpreserving(V_d)$ from that of $\cM^{\theta_d}(V_d, \ell_{V_d})$ in a stable range.  The result is a short exact sequence
  \begin{equation}\label{eq:23}
    0 \lra H^*(\cM^\orientationpreserving(V_d);\bQ) \lra A \overset{d_3}\lra A \lra 0,
  \end{equation}
  where $d_3: A \to A$ is the unique derivation satisfying $d_3(\kappa_{t^n c}) = n \kappa_{t^{n-1} c}$.  (The result is a scalar when $|t^{n-1} c| = 6$; this scalar is a characteristic number of $V_d$ and therefore the derivation $d_3$ depends on the degree $d$.)
\end{enumerate}

\subsubsection{Algebraic topology of $V_{d}$}
By the Lefschetz hyperplane theorem the inclusion $i: V_{d} \to \bC\bP^{4}$ is 3-connected. This first implies that $V_{d}$ is simply-connected. Writing $H^*(\bC\bP^{4};\bZ) = \bZ[x]/(x^{5})$ for $x=c_1(\Oo(1))$ and $t=i^*(x)$, we have
$$H^0(V_{d};\bZ) = \bZ \quad\quad H^1(V_{d};\bZ)=0 \quad\quad H^2(V_{d};\bZ) = \bZ\{t\}$$
and hence by Poincar{\'e} duality we have
$$H^4(V_{d};\bZ) = \bZ\{s\} \quad\quad H^5(V_{d};\bZ)=0 \quad\quad H^6(V_{d};\bZ) = \bZ\{u\}$$
where $\langle [V_{d}], u \rangle=1$ and $s \cdot t = u$. We also have $\langle i_*[V_{d}], x^3 \rangle=d$, obtained by intersecting $V_{d}$ with a generic $\bC\bP^1 \subset \bC\bP^{4}$, giving $t^3 = d \cdot u$ and hence $t^2 = d \cdot s$. By definition, $V_{d}$ is the zero locus of a transverse section of $\Oo(d) \to \bC\bP^{4}$, so its normal bundle in $\bC\bP^{4}$ is $i^*(\Oo(d))$ and hence as complex vector bundles we have
$$TV_{d} \oplus i^*(\Oo(d)) \oplus \underline{\bC} = i^*(T\bC\bP^{4}) \oplus \underline{\bC} =  i^*(\Oo(1))^{\oplus 5}$$
so taking total Chern classes yields $c(TV_{d}) = i^*\left(\frac{(1+x)^{5}}{(1+d x)}\right)$. We can therefore extract
\begin{align*}
c_3(TV_{d}) &= \left(\binom{5}{3} - \binom{5}{2}d + \binom{5}{1}d^2 -d^3\right)t^3
\end{align*}
and so compute the Euler characteristic of $V_{d}$ as $\langle [V_{d}], c_3(TV_{d}) \rangle$ to be
$$\chi(V_{d})  = d \cdot(10-10d+5d^2-d^3).$$
We therefore find that $H^3(V_{d};\bZ)$ is free of rank $4-\chi(V_{d}) = d^4-5d^3+10d^2-10d+4$, which finishes our calculation of the cohomology of $V_{d}$.

For later use we record two further characteristic classes of $V_{d}$, namely
\begin{align*}
w_2(TV_{d}) &= (5-d)t \mod 2\\
p_1(TV_{d}) &= (5-d^2)t^2,
\end{align*}
obtained from the identities $w_2 \equiv c_1 \mod 2$ and $p_1 = c_1^2-2c_2$ among characteristic classes of complex vector bundles.

\subsubsection{Genus of $V_{d}$}\label{sec:genus-v_d}
The genus of $V_d$ may be estimated from below by the methods of \S\ref{sec:disc-moduli-spac}, but in this case it turns out that an exact formula is possible.
It is a theorem of Wall \cite{WallClassV} that any simply-connected smooth 6-manifold $W$ has a decomposition $W \cong M \# g(S^3 \times S^3)$ with $H_3(M;\bZ)=0$, and so the genus of such a $W$ is given by half its third Betti number. Thus we have
$$g(V_{d}) = \frac{1}{2} (d^4-5d^3+10d^2-10d+4).$$
Similar formulae are known for higher dimensional smooth complex complete intersection varieties, see e.g.\ \cite{wood1975removing, morita1975kervaire, browder1979complete}.

\begin{remark}
More generally, if $\cL$ is an ample line bundle over a smooth projective complex manifold $M$ of complex dimension $n+1$ then for all $d \gg 0$ we may consider the smooth manifolds $U_d$ arising as the zeroes of generic holomorphic sections of $\cL^{\otimes d}$. Writing $x := c_1(\cL)$, as $\cL$ is ample we have $N := \int_M x^{n+1} \neq 0$. Writing $i : U_d \hookrightarrow M$ for the inclusion, and using that $i_*[U_d]$ is Poincar{\'e} dual to $e(\cL^{\otimes d}) = dx$, it follows that $\int_{U_d} i^*(x)^n = d \cdot N \neq 0$. The analogue of the calculation above gives that $c(TU_d) = i^*(\frac{c(TM)}{(1+dx)})$ and hence we have $\chi(U_d) = (-1)^n d^{n+1} N +O(d^n)$ and so $b_n = d^{n+1} N +O(d^n)$. If $n$ is odd then by the discussion in \S\ref{sec:disc-moduli-spac} we  have
$$g(U_d) = \frac{1}{2} d^{n+1} N +O(d^n).$$
 If $n$ is even, then the analogous calculation with the total Hirzebruch $\cL$-class gives that $\cL(TU_d) = i^*(\frac{\cL(TM)}{dx/\tanh(dx)})$ so $ \sigma(U_d) = \frac{2^{n+2}(2^{n+2}-1) B_{n+2}}{(n+2)!} d^{n+1} N +O(d^n)$, where $B_i$ denote the Bernoulli numbers. Hence, by the discussion in \S\ref{sec:disc-moduli-spac}, we have
 $$g(U_d) = \frac{1}{2}\left(1-\frac{2^{n+2}(2^{n+2}-1) |B_{n+2}|}{(n+2)!}\right) d^{n+1} N +O(d^n).$$
 The term $\frac{2^{n+2}(2^{n+2}-1) |B_{n+2}|}{(n+2)!}$ does not matter much for large $n$: the fact that the Taylor series for $\tanh(z)$ has convergence radius $\pi/2$ implies that that term is asymptotically smaller than $(2/\pi)^{n+\epsilon}$ as $n \to \infty$, for any $\epsilon > 0$; in particular it quickly becomes much smaller than 1.  In the relevant cases $n \geq 4$ it is at most $2/15$.
\end{remark}

\subsubsection{Moore--Postnikov 3-stage of $V_{d}$}
Let us write
$$\tau: V_{d} \overset{\ell_{V_d}}\lra B_{d} \overset{\theta_{d}}\lra B\SO(6)$$
for the Moore--Postnikov 3-stage of a map $\tau$ classifying the oriented tangent bundle of $V_{d}$, so $\ell_{V_d}$ is 3-connected and $\theta_{d}$ is 3-co-connected. From this we easily calculate the homotopy groups of $B_{d}$, as
\begin{align*}
0=\pi_1(V_{d}) \overset{\sim}\lra &\pi_1(B_{d})\\
\bZ=\pi_2(V_{d}) \overset{\sim}\lra &\pi_2(B_{d})\\
 &\pi_3(B_{d}) \overset{\sim}\lra \pi_3(B\SO(6))=0\\
 &\pi_i(B_{d}) \overset{\sim}\lra \pi_i(B\SO(6)) \text{ for all } i \geq 4.
\end{align*}
To understand the map $\theta_{d}$ on homotopy groups, it remains to understand the composition
\begin{equation}\label{eq:2ndSWVd}
\tau_* : \bZ=\pi_2(V_{d}) \overset{\sim}\lra \pi_2(B_{d}) \lra \pi_2(B\SO(6)) = \bZ/2.
\end{equation}
The latter group is detected by the Stiefel--Whitney class $w_2$, so this map is non-zero if and only if the class $w_2(TV_{d}) \in H^2(V_{d};\bZ/2) \cong \mathrm{Hom}(\pi_2(V_{d}), \bZ/2)$ is non-zero. We have seen that $w_2(TV_{d}) = (5-d)t$, so \eqref{eq:2ndSWVd} is surjective if and only if $d$ is even.

Let us abuse notation by writing ${t} \in H^2(B_{d};\bZ)$ for the unique class which pulls back to $t$ along $\ell_{V_d}$. If $d$ is even, then ${t}$ satisfies ${t} \equiv w_2(\theta_{d}^*\gamma) \mod 2$.
Thus there is a $\mathrm{Spin^c}$-structure on the bundle $\theta_d^*\gamma$ with $c_1={t}$, and choosing one provides a commutative diagram
\begin{equation*}
\begin{tikzcd} 
B_{d} \arrow[rd, "\theta_{d}"] \arrow[rr, "f"]& & B\mathrm{Spin^c}(6) \arrow[ld]\\
 & B\mathrm{SO}(6).
\end{tikzcd}
\end{equation*}
It may be directly checked using the above calculations that the map $f$ induces an isomorphism on all homotopy groups, so is a weak equivalence (over $B\mathrm{SO}(6)$).

If $d$ is odd then we have $w_2(\theta_{d}^*\gamma)=0$, so we may choose a $\mathrm{Spin}$-structure on the bundle $\theta_{d}^*\gamma$, which provides a commutative diagram
\begin{equation*}
\begin{tikzcd} 
B_{d} \arrow[rd, "\theta_{d}"] \arrow[rr, "h"]& & B\mathrm{Spin}(6) \arrow[ld]\\
 & B\mathrm{SO}(6).
\end{tikzcd}
\end{equation*}
It may be directly checked using the above calculations that the map $h \times {t} : B_{d} \to B\mathrm{Spin}(6) \times K(\bZ, 2)$ induces an isomorphism on all homotopy groups, so is a weak equivalence  (over $B\mathrm{SO}(6)$).

In either case, the map
$$\theta_{d} \times t : B_{d} \lra B\mathrm{SO}(6) \times K(\bZ,2)$$
is a rational homotopy equivalence (over $B\mathrm{SO}(6)$), so we have
$$H^*(B_{d};\bQ) = \bQ[t, p_1, p_2, e].$$
Writing as in the previous examples $\cB$ for the set of monomials in $p_1$, $p_2$, and $e$, by Theorem \ref{thm:main-cohomological} the map
$$\bQ[\kappa_{t^i c} \, | \, c \in \cB, i \geq 0, |c|+2i > 6] \lra H^*(\mathcal{M}^{\theta_{d}}(V_{d},\ell_{V_d});\bQ)$$
is an isomorphism in degrees $* \leq \frac{d^4-5d^3+10d^2-10d+4}{4}$, establishing \eqref{eq:45}.

\subsubsection{Change of tangential structure}\label{sec:ChangeOfStr}

We wish to use the above to compute the rational cohomology of $\mathcal{M}^\orientationpreserving(V_{d})$ in a range of degrees, so must analyse the forgetful map
$\mathcal{M}^{\theta_{d}}(V_{d}, \ell_{V_d}) \to \mathcal{M}^\orientationpreserving(V_{d})$. We shall do this in two stages, given by the maps of tangential structures
\begin{equation*}
\begin{tikzcd} 
B_{d} \arrow[rrd,',"\theta_{d}"] \arrow[rr, "u = \theta_{d} \times t"]& &  B\mathrm{SO}(6) \times K(\bZ,2) \arrow[d, "\mu"] \arrow[rr, "\mu"]& & B\mathrm{SO}(6) \arrow[lld, "\mathrm{Id}"] \\
 & & B\mathrm{SO}(6).
\end{tikzcd}
\end{equation*}

The space of $\theta_{d}$-structures on $V_{d}$ refining the $\mu$-structure $u \circ \ell_{V_d}$ is homotopy equivalent to the space of lifts
\begin{equation*}
\begin{tikzcd} 
 & B_{d} \arrow[d, "u"]\\
V_{d} \arrow{r}[swap]{u \circ \ell_{V_d}} \arrow[ru, dashed]& B\mathrm{SO}(6) \times K(\bZ,2),
\end{tikzcd}
\end{equation*}
and $\pi_0(\hAut(u))$ acts on the set of homotopy classes of such lifts. If $G \leq \hAut(u)$ is the submonoid of those path components that preserve the $\theta_d$-structure $\ell_{V_d}$ up to diffeomorphisms of $V_d$ preserving the $\mu$-structure $u \circ \ell_{V_d}$, then there is a fibration sequence
$$\mathcal{M}^{\theta_{d}}(V_{d}, \ell_{V_d}) \lra \mathcal{M}^{\mu}(V_{d}, u \circ \ell_{V_d}) \lra BG.$$
By the discussion in Remark \ref{rem:Kreck}, $G = \hAut(u)_{[V_d, \ell_{V_d}]}$ as long as $g(V_d,\ell_{V_d}) \geq 3$, and this fibration sequence is an instance of \eqref{eq:22}.

As we have seen above, the map $u$ is a rational homotopy equivalence, and it is immediate from this that $\pi_i(\hAut(u)) \otimes \bQ=0$ for $i>0$, so $G$ has no higher rational homotopy groups. 

We claim that $\pi_0(G)$ is also trivial, and in fact we shall show that $\pi_0(\hAut(u))$ is trivial (so $G=\hAut(u)$). To see this, let $\phi \in \hAut(u)$, and we must then show that the following lifting problem admits a solution:
\begin{equation*}
\begin{tikzcd} 
 {\partial [0,1]} \times B_{d} \arrow[rr, "\mathrm{Id} \sqcup \phi"] \arrow[d]& & B_{d} \arrow[d, "u"]\\
 {[0,1]} \times B_{d} \arrow[r, "proj"] \arrow[rru, dashed] & B_{d} \arrow[r, "u"]& B\mathrm{SO}(6) \times K(\bZ,2).
\end{tikzcd}
\end{equation*}
By consideration of the cases $B_{d} = B\mathrm{Spin^c}(6)$ and $B_{d} = B\mathrm{Spin}(6) \times K(\bZ,2)$, we see that the homotopy fibre of $u$ is a $K(\bZ/2,1)$, so there is a unique obstruction to finding the required lift, lying in 
$$H^2([0,1] \times B_{d}, \partial [0,1] \times B_{d};\bZ/2) \cong {H}^1(B_{d};\bZ/2)=0.$$
It follows that $BG$ is simply-connected and has trivial higher rational homotopy groups, so $\mathcal{M}^{\theta_{d}}(V_{d}, \ell_{V_d}) \to \mathcal{M}^{\mu}(V_{d}, u \circ \ell_{V_d})$ is a rational homotopy equivalence.

Analogously to the above, if $H \leq \hAut(\mu)$ is the submonoid of those path components that preserve the $\mu$-structure $u \circ \ell = \tau \times t$ up to orientation-preserving diffeomorphism of $V_d$, then there is a fibration sequence
$$\mathcal{M}^{\mu}(V_{d}, u \circ \ell_{V_d}) \lra \mathcal{M}^\orientationpreserving(V_{d}) \lra BH.$$
Again, by Remark \ref{rem:Kreck}, $H = \hAut(\mu)_{[V_d, u \circ \ell_{V_d}]}$ as long as $g(V_d, u \circ \ell_{V_d}) \geq 3$, and this fibration sequence is an instance of \eqref{eq:22}. As the fibration $\mu$ is trivial, we have
\begin{align*}
\hAut(\mu) &\simeq \map(B\mathrm{SO}(6), \hAut(K(\bZ,2)))\\
&\simeq \bZ^\times \ltimes \map(B\mathrm{SO}(6), K(\bZ,2))\\
&\simeq \bZ^\times \ltimes K(\bZ,2).
\end{align*}
The non-trivial path component of this monoid acts on $H^2(B\mathrm{SO}(6) \times K(\bZ,2);\bZ) = \bZ\{t\}$ as $-1$, but any orientation-preserving diffeomorphism of $V_d$ fixes $t^3 \in H^6(V_d;\bZ)$ so acts as $+1$ on $H^2(V_d;\bZ) = \bZ\{t\}$. Thus the non-trivial path component of $\hAut(\mu)$ does not lie in $H$, so $H \simeq K(\bZ,2)$. Thus the fibration sequence is of the form
$$\mathcal{M}^{\mu}(V_{d}, u \circ \ell_{V_d}) \lra \mathcal{M}^\orientationpreserving(V_{d}) \lra K(\bZ,3).$$
The Serre spectral sequence for this fibration, in rational cohomology, has two columns and so a single possible non-zero differential. In the stable range, using the above, it has the form
$$E_2^{*,*} = \Lambda[\iota_3] \otimes \bQ[\kappa_{t^i c} \, | \, c \in \cB, i \geq 0, |c|+2i > 6] \Longrightarrow H^*(\mathcal{M}^\orientationpreserving(V_{d});\bQ).$$
It remains to determine the $d_3$-differential, which by the Leibniz rule is done by the following lemma.

\begin{lemma}\label{lem:VdIdentifyDiff}
We have $d_3(\kappa_{t^n c}) = \iota_3 \otimes n \cdot \kappa_{t^{n-1} c}$.
\end{lemma}
\begin{proof}
We have $d_3(\kappa_{t^n c}) = \iota_3 \otimes x$ for some $x$. The action map
$$a: K(\bZ,2) \times \mathcal{M}^{\mu}(V_{d}, u \circ \ell_{V_d}) \lra \mathcal{M}^{\mu}(V_{d}, u \circ \ell_{V_d})$$
classifies the following data: the $V_{d}$-bundle
$$\pi : K(\bZ,2) \times \mathcal{E}^{\mu}(V_{d}, u \circ \ell_{V_d}) \to K(\bZ,2) \times \mathcal{M}^{\mu}(V_{d}, u \circ \ell_{V_d})$$
pulled back by projection to the second factor, equipped with the $\mu$-structure
$$K(\bZ,2) \times \mathcal{E}^{\mu}(V_{d}, u \circ \ell_{V_d}) \overset{\tau \times \tilde{t}}\lra B\mathrm{SO}(6) \times K(\bZ,2)$$
where $\tau$ is given by projection to $\mathcal{E}^{\mu}(V_{d}, u \circ \ell_{V_d})$ and its vertical tangent bundle, and $\tilde{t} = \iota_2 \otimes 1 + 1 \otimes t$.

The class $x$ is related to this action by the formula
$$a^*(\kappa_{t^n c}) = 1 \otimes \kappa_{t^n c} + \iota_2 \otimes x + \cdots.$$
Using the description above we calculate $a^*(\kappa_{t^n c})$ as
$$\pi_!((\iota_2 \otimes 1 + 1 \otimes t)^n \cdot \tau^*(c)) = \pi_!\left(\sum_{i=0}^n \binom{n}{i} \iota_2^{i} \otimes (t^{n-i} \cdot \tau^*c)\right)$$
and the K{\"u}nneth factor in $H^2(K(\bZ,2);\bZ) \otimes H^{|\kappa_{t^n c}|-2}(\mathcal{M}^{\mu}(V_{d}, u \circ \ell_{V_d});\bZ)$ is $\iota_2 \otimes (n \cdot \kappa_{t^{n-1} c})$. It follows that $x = n \cdot \kappa_{t^{n-1} c}$, as required.
\end{proof}

It follows from this lemma that the differential $d_3$ is a surjection from the first column to the third column, so that
$$H^*(\mathcal{M}^\orientationpreserving(V_{d});\bQ) \cong \mathrm{Ker}(d_3 \circlearrowright \bQ[\kappa_{t^n c} \, | \, c \in \cB, i \geq 0, |c|+2n > 6])$$
in degrees $* \leq \frac{d^4-5d^3+10d^2-10d+4}{4}$, establishing~(\ref{eq:23}).

It may at first appear that this ring does not depend on ${d}$, but this formula is to be understood carefully. If $|\kappa_{t^n c}|=2$ then $d_3(\kappa_{t^n c}) \in \bQ$ is a scalar, and must be evaluated: this is a \emph{boundary condition} for the derivation $d_3$, and is a characteristic number of $V_{d}$. The $\kappa_{t^n c}$ of degree 2 are given by the $t^i c$ of degree 8, so are $p_2$, $p_1^2$, $te$, $t^2 p_1$, and $t^4$, and these have
\begin{align*}
d_3(\kappa_{p_2}) &= 0\\
d_3(\kappa_{p_1^2}) &= 0\\
d_3(\kappa_{te}) &= \kappa_e = \chi(V_{d}) = d \cdot(10-10d+5d^2-d^3)\\
d_3(\kappa_{t^2 p_1}) &= 2 \kappa_{tp_1} = 2 d (5-d^2)\\
d_3(\kappa_{t^4}) &= 4 \kappa_{t^3} = 4\cdot d.
\end{align*}
For the penultimate one we use the calculation of the first Pontryagin class of $V_d$. As an example, $H^2(\mathcal{M}^\orientationpreserving(V_{d});\bQ)$ is 4-dimensional and is spanned by the classes
\begin{align*}
\kappa_{p_2}, \quad \kappa_{p_1^2}, \quad \kappa_{te} - \frac{10-10d+5d^2-d^3}{4}\kappa_{t^4}, \quad \text{ and } \quad  \kappa_{t^2 p_1} - \frac{5-d^2}{2}\kappa_{t^4}.
\end{align*}

\begin{remark}
In the Serre spectral sequence for the fibration $\pi: \mathcal{E}^\orientationpreserving(V_{d}) \to \mathcal{M}^\orientationpreserving(V_{d})$ modelling the universal oriented $V_g$-bundle as in \eqref{eq:25}, the class $t \in H^2(V_{d};\bQ) = E_2^{0,2}$ must be a permanent cycle. (This may be seen as the Euler class of the vertical tangent bundle $T_\pi\mathcal{E}^\orientationpreserving(V_{d})$ restricts to $e(TV_{d}) \in H^6(V_g;\bQ)$, so this must be a permanent cycle, and this is a non-zero multiple of $t^3$. As $d_3(t^3) = 3 t^2 \cdot d_3(t)$, if $d_3(t) \neq 0$ then $t^3$ would not by a permanent cycle, a contradiction.) Thus there exists a class $\bar{t} \in H^2(\mathcal{M}^\orientationpreserving(V_{d});\bQ)$ restricting to $t \in H^2(V_{d};\bQ)$. We may therefore construct the class $\kappa_{\bar{t}^n c} := \pi_!(\bar{t}^n c) \in H^*(\mathcal{M}^\orientationpreserving(V_{d});\bQ)$.

However, the class $\bar{t}$ is not uniquely determined by the above discussion: if $\delta t \in H^2(\mathcal{M}^\orientationpreserving(V_{d});\bQ)$ is any class then $\bar{\bar{t}} =\bar{t} + \pi^*(\delta t)$ is another possible choice, and we then have 
$$\kappa_{\bar{\bar{t}}^n c} = \pi_!((\bar{t} + \pi^*(\delta t))^n c) = \kappa_{\bar{t}^n c} + (\delta t)(n \cdot \kappa_{\bar{t}^{n-1} c}) + (\delta t)^2  \cdots,$$
a potentially different cohomology class.

By the formula in Lemma \ref{lem:VdIdentifyDiff}, we may think of the derivation $d_3$ as being $\frac{\partial}{\partial t}$. From this point of view the polynomials in the classes $\kappa_{t^n c}$ that lie in the kernel of $d_3 = \frac{\partial}{ \partial t}$ are precisely those that are independent of the choice of $\bar{t}$ when evaluated in $H^*(\mathcal{M}^\orientationpreserving(V_{d});\bQ)$ as described above.
\end{remark}

\subsection{Another $\Spin^c(6)$ example}\label{sec:OtherRat}

For a simply-connected manifold $W$ of dimension $2n \geq 6$, the formula of Theorem \ref{thm:general-str} for the homology of $\cM^\orientationpreserving(W)$ in a range of degrees seems at first glance as though it only depends on the equivariant homotopy type of the $\GL_{2n}(\bR)$-space $\Theta$ having an $n$-co-connected equivariant map $u : \Theta \to \bZ^\times$ and an $n$-connected equivariant map $\str_W : \Fr(TW) \to \Theta$. However, the codomain of the map \eqref{eq:11} is the disconnected space $(\Omega^\infty MT\Theta)\moddd \hAut(u)$, and the different path-components of this space can have different cohomology, even rationally. In this section we give an example of this behaviour.

\begin{construction}
Let $V \to S^2$ be the unique non-trivial 5-dimensional real vector bundle, and $M = S(V)$ be its sphere bundle; it is an $S^4$-bundle over $S^2$ with the same homology as $S^4 \times S^2$. If we write $\pi : M \to S^2$ for the bundle projection, then there is an isomorphism $TM \cong \pi^*(V)\oplus \epsilon^1$. In particular, the $\Spin^c$-structure on $V$ given by a generator of $H^2(S^2;\bZ)$ gives one on $M$ (which is $\Spin^c$-nullbordant), and the corresponding map $\ell_M : M \to B\Spin^c(6)$ is 3-connected. This induces a $\Spin^c$-structure on $M_g := M \# g (S^3 \times S^3)$ such that $\ell_{M_g} : M_g \to B\Spin^c(6)$ is also 3-connected.
\end{construction}

Let $\theta : B\Spin^c(6) \to B\SO(6)$. As in the last section, if $K \leq \hAut(\theta)$ is the submonoid of those path components that stabilise the $\theta$-structure $\ell_{M_g}$ up to diffeomorphism of $M_g$, then there is a fibration sequence
$$\mathcal{M}^{\theta}(M_g,\ell_{M_g}) \lra \mathcal{M}^\orientationpreserving(M_g) \lra BK.$$

\begin{lemma}
We have $\hAut(\theta) \simeq \bZ^\times \ltimes K(\bZ,2)$.
\end{lemma}
\begin{proof}
The fibration $\theta : B\Spin^c(6) \to B\SO(6)$ has fibre $K(\bZ,2)$, and is principal, so there is an action of $K(\bZ,2)$ on $B\Spin^c(6)$ fibrewise over $B\SO(6)$. Furthermore, writing $\Spin^c(6) = \Spin(6) \times_{\bZ^\times} \U(1)$ we see that complex conjugation on the $\U(1)$ factor gives an involution $c$ of $B\Spin^c(6)$ over $B\SO(6)$. Together these give a map $\bZ^\times \ltimes K(\bZ,2) \to \hAut(\theta)$ which can be shown to be an equivalence by obstruction theory just as in \S\ref{sec:ChangeOfStr} or the proof of Lemma \ref{lem:WgIdhAut}.
\end{proof}

\begin{lemma}\label{lem:M_ghAut}
We have $K = \hAut(\theta)$.
\end{lemma}
Let us give two proofs of this lemma, one in terms of the manifolds themselves, and one using the infinite loop spaces of the relevant Thom spectrum.
\begin{proof}
The proof of the previous lemma shows that if $\ell$ and $\ell'$ are two $\theta$-structures on $M_g$ then there is a unique obstruction to them being homotopic, namely
$$\ell^*(t) - (\ell')^*(t) \in H^2(M_g;\bZ).$$
We therefore see that $\ell_{M_g}$ and $c\circ \ell_{M_g}$, where $c$ is the involution of $B\Spin^c(6)$ over $B\SO(6)$, are not fibrewise homotopic as the obstruction is $2\ell_{M_g}^*(t) \neq 0 \in H^2(M_g;\bZ)$. 

However, pulling back the vector bundle $V \to S^2$ along a diffeomorphism of $S^2$ of degree $-1$ gives an isomorphic oriented vector bundle, as $\pi_2(B\SO(5))= \bZ/2$, and so this degree $-1$ diffeomorphism is covered by a diffeomorphism $M \to M$ which acts as $-1$ on $H^2(M;\bZ)$ and as $+1$ on $H^4(M;\bZ)$, so is orientation-reversing. Composing this with the fibrewise antipodal map of $\pi : M \to S^2$ gives a diffeomorphism $\varphi: M \to M$ which acts as $-1$ on both $H^2(M;\bZ)$ and $H^4(M;\bZ)$, so is orientation-preserving: we may then isotope it to fix a disc, and hence extend it to a diffeomorphism $\varphi_g : M_g \to M_g$ acting as $-1$ on $H^2(M;\bZ)$ and on $H^4(M;\bZ)$.

Now the $\theta$-structures $\ell_{M_g} \circ D\varphi_g$ and $c \circ \ell_{M_g}$ on $M_g$ are homotopic, as
$$(\ell_{M_g} \circ D\varphi_g)^*(t) = \varphi_g^*(\ell_{M_g}^*(t)) = - \ell_{M_g}^*(t) = \ell_{M_g}^*(-t) = (c \circ \ell_{M_g})^*(t).$$
This shows that $c \in \hAut(\theta)$ lies in the submonoid $K$, as it preserves the $\theta$-structure $\ell_{M_g}$ up to a diffeomorphism of $M_g$.
\end{proof}

\begin{proof}[Alternative proof]
By the discussion in Remark \ref{rem:Kreck}, the submonoid $K \leq \hAut(\theta)$ agrees with $\hAut(\theta)_{[M_g, \ell_{M_g}]}$ as long as $g \geq 3$, so is the stabiliser of $\alpha(M_g, \ell_{M_g}) \in \pi_0(MT\Spin^c(6))$.

Thomifying the map $B\Spin^c(6) \to B\Spin^c$ gives a fibre sequence of spectra
$$F \lra MT\Spin^c(6) \lra \Sigma^{-6}M\Spin^c$$
and it is easy to check that $F$ is connective and has $\pi_0(F) \cong \bZ$. We therefore have an exact sequence
$$\pi_0(F) \cong \bZ \lra \pi_0(MT\Spin^c(6)) \lra \pi_6(M\Spin^c) = \Omega_6^{\Spin^c},$$
and the left-hand map can be seen to send a generator to $\alpha(S^6, \ell_{S^6})$, where $\ell_{S^6}$ is the unique $\Spin^c(6)$-structure on $S^6$ compatible with its orientation.

As the $\Spin^c(6)$-manifold $M$ is constructed as the sphere bundle of a $\Spin^c$ vector bundle, its class is trivial in $\Omega_6^{\Spin^c}$ as it bounds the associated disc bundle; similarly for $M_g = M \# g(S^3 \times S^3)$. Thus $\alpha(M_g, \ell_{M_g})$ is a multiple of $\alpha(S^6, \ell_{S^6})$ (by taking Euler characteristic we see that it is $2-g$ times it) and so is fixed by $\hAut(\theta)$, as the $\Spin^c(6)$-structure on $S^6$ is unique given its orientation.
\end{proof}

We may therefore develop the following diagram of fibration sequences
\begin{equation*}
\begin{tikzcd} 
\mathcal{M}^{\theta}(M_g, \ell_{M_g}) \ar[r] \ar[equal]{d} & X_g \ar[r] \ar[d] & K(\bZ,3) \ar[d]\\
\mathcal{M}^{\theta}(M_g, \ell_{M_g}) \ar[r] \ar[d]  & \mathcal{M}^\orientationpreserving(M_g) \ar[r] \ar[d] & B(\bZ^\times
 \ltimes K(\bZ,2)) \ar[d]\\
* \ar[r] & B\bZ^\times \ar[equal]{r} & B\bZ^\times,
\end{tikzcd}
\end{equation*}
whose middle row is the fibration sequence, with lower middle arrow defined to make the bottom right-hand square commute, and $X_g$ as its homotopy fibre, and top right-hard square homotopy cartesian.

The calculation of the previous section applies to the top row, showing that
$$H^*(X_g;\bQ) = \mathrm{Ker}(d_3 \circlearrowright \bQ[\kappa_{t^n c} \, | \, c \in \cB, i \geq 0, |c|+2n > 6])$$
in a stable range, this time subject to the boundary conditions
\begin{align*}
d_3(\kappa_{te}) &= \kappa_e = \chi(M_{g}) = 4-2g\\
d_3(\kappa_{t^2 p_1}) &= 2 \kappa_{tp_1} = 0\\
d_3(\kappa_{t^4}) &= 4 \kappa_{t^3} = 0.
\end{align*}
However, now the Serre spectral sequence for the middle column gives the calculation
$$H^*(\mathcal{M}^\orientationpreserving(M_g);\bQ) = H^*(X_g;\bQ)^{\bZ^\times} = \mathrm{Ker}(d_3 \circlearrowright \bQ[\kappa_{t^n c} \, | \, c \in \cB, i \geq 0, |c|+2n > 6])^{\bZ^\times}$$
in a stable range, where the invariants are taken with respect to the involution $t \mapsto -t$.

Let us explain something of the structure of this ring in low degrees.  In particular, we shall see that unlike the previous examples it not a free graded-commutative algebra, even in the stable range where our formulae apply.
Before taking $\Z^\times$-invariants, in degree 2 the kernel is spanned by $\{\kappa_{p_2}, \kappa_{p_1^2}, \kappa_{t^2 p_1}, \kappa_{t^4}\}$, and these classes are all fixed by the involution, giving
$$\dim_\bQ H^2(\mathcal{M}^\orientationpreserving(M_g);\bQ) = 4.$$
In degree 4 the kernel of $d_3$ is 16 dimensional, spanned by the 10-dimensional vector space $\mathrm{Sym}^2(\bQ\{\kappa_{p_2}, \kappa_{p_1^2}, \kappa_{t^2 p_1}, \kappa_{t^4}\})$ along with the classes
\begin{align*}
& \kappa_{te}\kappa_{p_2} - (4-2g)\kappa_{tp_2}\\
& \kappa_{te}\kappa_{p_1^2} - (4-2g)\kappa_{tp_1^2}\\
& (4-2g)\kappa_{t^3 p_1} - 3\kappa_{t^2p_1}\kappa_{te}\\
& (4-2g)\kappa_{t^5} - 5\kappa_{t^4}\kappa_{te}\\
& \kappa_{p_1e}\\
& \kappa_{te}^2 - (4-2g) \kappa_{t^2e}.
\end{align*}
Of these, the last two classes are invariant under the involution while first four are \emph{anti-}invariant, and hence
$$\dim_\bQ H^4(\mathcal{M}^\orientationpreserving(M_g);\bQ) = 12.$$

In higher degrees, we find that even though, for example, the class $\kappa_{te}\kappa_{p_2} - (4-2g)\kappa_{tp_2}$ is not invariant under the involution, its square is invariant and therefore defines a class in $H^8(\mathcal{M}^\orientationpreserving(M_g);\bQ)$. Similarly with products of any two classes that are anti-invariant and in the kernel of $d_3$. In degree 16 we find the relation
\begin{align*}
&((\kappa_{te}\kappa_{p_2} - (4-2g)\kappa_{tp_2})(\kappa_{te}\kappa_{p_1^2} - (4-2g)\kappa_{tp_1^2}))^2\\
&\quad\quad = (\kappa_{te}\kappa_{p_2} - (4-2g)\kappa_{tp_2})^2 (\kappa_{te}\kappa_{p_1^2} - (4-2g)\kappa_{tp_1^2})^2
\end{align*}
among squares of classes of degree 8, showing that the ring is not free.

\section{Abelianisations of mapping class groups}\label{sec:AbCalc}

The theory described above may in principle be used for calculations in integral homology and cohomology, though this is of course far more difficult. In practice such calculations are restricted to low dimensions, and have a different flavour to those described in \S\ref{sec:RatCalc}. Here one must obtain information about the low-dimensional homology of $\Omega^\infty MT\Theta$, which is roughly the same as the low-dimensional homotopy of $\Omega^\infty MT\Theta$, which is the homotopy of the spectrum $MT\Theta$ in small positive degrees. But the spectrum $MT\Theta$ is non-connective, so computing its $\pi_i$ is comparable to computing $\pi_{i+2n}$ of a connective spectrum (in the alternative proof of Lemma \ref{lem:M_ghAut} we have already engaged with this a bit, though we avoided having to actually compute).

As an example of the kinds of calculations that one is required to make, and to give some ideas of the kinds of techniques that can be used to tackle them, in this section we shall survey the calculation in \cite{GR-WAb} of $H_1(\cM(W_{g,1});\bZ)$, and then describe analogous calculations for certain non-simply connected 6-manifolds.

Recall that for a manifold $W$, possibly with boundary, its \emph{mapping class group}\index{mapping class group} is
$$\Gamma_\partial(W) := \pi_0(\Diff_\partial(W)).$$
Equivalently, it is the fundamental group of $B\Diff_\partial(W)$, so by the Hurewicz theorem we may identify its abelianisation as 
$$\Gamma_\partial(W)^{ab} \cong H_1(B\Diff_\partial(W);\bZ) \cong H_1(\mathcal{M}(W);\bZ).$$

\subsection{The manifolds $W_{g,1}$}\label{sec:AbWgs}

We return to the $2n$-manifolds $W_{g,1}$ of \S\ref{sec:RatWgs}. Just as in that section, there is a map 
$$\cM(W_{g,1}) \simeq \cM^{\theta_n}(W_{g,1}, \ell_{W_{g,1}}) \lra \Omega^\infty MT\theta_n$$
that for $2n \geq 6$ is a homology isomorphism in degrees $\leq \frac{g-3}{2}$ onto the path component that it hits. In particular, as long as $g \geq 5$ we have isomorphisms
$$\Gamma_\partial(W_{g,1})^{ab} \cong H_1(\mathcal{M}(W_{g,1});\bZ) \cong H_1(\Omega^\infty_0 MT\theta_n ; \bZ) \cong \pi_1(MT\theta_n),$$
using that all path components of $\Omega^\infty_0 MT\theta_n$ are homotopy equivalent, that the Hurewicz map is an isomorphism (as this space is a loop space), and that $\pi_1$ of the space $\Omega^\infty_0 MT\theta_n$ is the same as that of the spectrum $MT\theta_n$.

In \cite{GR-WAb} we attempted to calculate this group, at least in terms of other standard groups arising in geometric topology. Here we shall summarise the results and general strategy of that paper, though we refer there for more details.

To state the main result, consider the bordism theory $\Omega_*^{\langle n \rangle}$ associated to the fibration $B\OO\langle n \rangle \to B\OO$ given by the $n$-connected cover, and represented by the spectrum $M\OO\langle n \rangle$ (cf.\ \cite[p.\ 51]{Stong}). The natural map $B\OO(2n)\langle n \rangle \to B\OO\langle n \rangle$ covering the stabilisation map $B\OO(2n) \to B\OO$ provides a map of spectra
$$s : MT\theta_n \lra \Sigma^{-2n} M\OO\langle n \rangle,$$
and on $\pi_1$ this gives a homomorphism $s_*: \pi_1(MT\theta_n) \to \Omega_{2n+1}^{\langle n \rangle}$. The group $\pi_1(MT\theta_n)$ is determined in terms of this as follows.

\begin{theorem}\label{thm:AbWgsThm}
There is an isomorphism
$$s_* \oplus f : \pi_1(MT\theta_n) \lra \Omega_{2n+1}^{\langle n \rangle} \oplus \begin{cases}
(\bZ/2)^2 & \text{ if $n$ is even}\\
0 & \text{ if $n$ is 1, 3, or 7}\\
\bZ/4 & \text{ else}
\end{cases}$$
for a certain homomorphism $f$.
\end{theorem}

Furthermore, the groups $\Omega_{2n+1}^{\langle n \rangle}$ are related to the stable homotopy groups of spheres as follows: there is a homomorphism\index{cokernel of $J$}
$$\rho': \mathrm{Cok}(J)_{2n+1} \lra \Omega_{2n+1}^{\langle n \rangle}$$
given by considering a stably framed manifold as a manifold with $B\OO\langle n \rangle$-structure, which is surjective and whose kernel is generated by the class of a certain homotopy sphere\index{homotopy sphere} $\Sigma_Q^{2n+1}$. In several cases it follows from work of Stolz that the class of $\Sigma_Q$ in $\mathrm{Cok}(J)$ is trivial---so $\rho'$ is an isomorphism---but this is not known in general. Combining the above with known calculations of $\Omega_{*}^{\langle 2 \rangle} = \Omega_{*}^{\langle 3 \rangle} = \Omega_{*}^{\Spin}$ and $\Omega_{*}^{\langle 4 \rangle} = \Omega_{*}^{\mathrm{String}}$ gives the following.

\begin{table}[h]
\centering
\label{table:1}
\begin{tabular}{c|c c c c c c c}
$n$              & 1 & 2 & 3 & 4 & 5  & 6 & 7 \\
\hline
$\pi_1(MT\theta_n)$ & 0 & $(\bZ/2)^2$ & 0 & $(\bZ/2)^4$ & $\bZ/4$ & $(\bZ/2)^2 \oplus \bZ/3$ & $\bZ/2$  \\
\end{tabular}
\end{table}

Let us outline the proof of Theorem \ref{thm:AbWgsThm} which was given in \cite{GR-WAb}, as we shall need to refer to details of this argument in the following section. The argument combines methods from (stable) homotopy theory with Theorem \ref{thm:homotopical} again. Starting with homotopy theory, we first let $F$ denote the homotopy fibre of the map of spectra $s : MT\theta_n \to \Sigma^{-2n} M\OO\langle n \rangle$, and construct a map $\Sigma^{-2n} \SO/\SO(2n) \to F$ which can be shown to be $n$-connected, for example by computing its effect on homology. On the other hand $\SO/\SO(2n)$ is $(2n-1)$-connected, so by Freudenthal's suspension theorem the map
$$\pi_{2n+1}(\SO/\SO(2n)) \lra \pi_{2n+1}^s(\SO/\SO(2n)) \cong \pi_1^s(\Sigma^{-2n} \SO/\SO(2n))$$
is an isomorphism for $n \geq 2$, and similarly for one homotopy group lower. It follows from a calculation of Paechter \cite{Paechter} that $\pi_{2n+1}(\SO/\SO(2n))$ is $(\bZ/2)^2$ if $n$ is even and is $\bZ/4$ if $n$ is odd, and also that $\pi_{2n}(\SO/\SO(2n)) \cong \bZ$. Putting the above together, we find an exact sequence
\begin{equation}\label{eq:SES}
\Omega_{2n+2}^{\langle n \rangle} \overset{\partial} \lra  \begin{cases}
(\bZ/2)^2 & \text{ if $n$ is even}\\
\bZ/4 & \text{ if $n$ is odd}
\end{cases} \lra \pi_1(MT\theta_n) \lra \Omega_{2n+1}^{\langle n \rangle} \lra \bZ.
\end{equation}
The rightmost map is zero (as its domain is easily seen to be a torsion group). In the cases $n \in \{1,3,7\}$ it can be shown that the images of $\bC\bP^2$, $\bH\bP^2$, and $\bO\bP^2$ under the leftmost map are non-zero modulo 2, so the leftmost map is surjective. In the remaining cases one must show that the leftmost map is zero, and that the resulting short exact sequence is split, via a homomorphism $f$ as in the statement of Theorem \ref{thm:AbWgsThm}.

At this point is is convenient to use the isomorphism $\Gamma_\partial(W_{g,1})^{ab} \cong \pi_1(MT\theta_n)$ for some $g \gg 0$. The action of $\Gamma_\partial(W_{g,1})$ on $H_n(W_{g,1};\bZ)$ respects the $(-1)^n$-symmetric intersection form $\lambda$, and if $n \neq 1$, $3$, or $7$ then it also respects a certain quadratic refinement $\mu$ of this bilinear form. This yields a homomorphism
$$\Gamma_\partial(W_{g,1}) \lra \mathrm{Aut}(H_n(W_{g,1};\bZ), \lambda, \mu).$$
These automorphism groups have been studied by other authors, and their abelianisations have been identified (for $g \gg 0$) as $(\bZ/2)^2$ if $n$ if even or $\bZ/4$ if $n$ is odd. A careful analysis of the maps involved shows that the resulting homomorphism
$$f : \pi_1(MT\theta_n) \cong \Gamma_\partial(W_{g,1})^{ab} \lra \mathrm{Aut}(H_n(W_{g,1};\bZ), \lambda, \mu)^{ab} \cong \begin{cases}
(\bZ/2)^2 & \text{ if $n$ is even}\\
\bZ/4 & \text{ if $n$ is odd}
\end{cases}$$
splits the short exact sequence arising from \eqref{eq:SES}, as required.

\begin{remark}
The Pontryagin dual of the finite abelian group $H_1(\cM(W_{g,1});\bZ)$ calculated here is the torsion subgroup of $H^2(\cM(W_{g,1});\bZ)$. The torsion free quotient of the latter group has been analysed in detail by Krannich and Reinhold \cite{KR}. The (unknown, at present) order of the element $[\Sigma_Q] \in \mathrm{Cok}(J)_{2n+1}$ arises there too.
\end{remark}

\subsection{Some non-simply-connected 6-manifolds}

For the example discussed in the previous section the theory described above is not the only way to calculate $\Gamma_\partial(W_{g,1})^{ab}$, because Kreck \cite{KreckAut} has described the groups $\Gamma_\partial(W_{g,1})$ up to two extension problems, and Krannich \cite{KrannichMCG} has recently resolved these extensions completely for $n$ odd, and determined enough about them to calculate $\Gamma_\partial(W_{g,1})^{ab}$ for all $n \geq 3$ and all $g \geq 1$. However, for even slightly more complicated manifolds such an alternative approach is not available, and we suggest that the theory described above is the best way to approach the calculation of $\Gamma_\partial(W)^{ab}$. In this section we illustrate this with an example which seems inaccessible by other means.

Let $G$ be a virtually polycyclic group,\index{virtually polycyclic} and consider a compact 6-manifold $W$ such that
\begin{enumerate}[(i)]
\item a map $\tau_W$ classifying the tangent bundle of $W$ admits a lift $\ell_W$ along
$$\theta : B\Spin(6) \times BG \overset{\mathrm{pr}_1}\lra B\Spin(6) \lra B\OO(6)$$
such that $\ell_W : W \to B\Spin(6) \times BG$ is 3-connected, and

\item $(W, \partial W)$ is 2-connected.
\end{enumerate}
Such manifolds exist for any virtually polycyclic $G$: these groups satisfy Wall's \cite{WallFin} finiteness condition ($F$) by \cite[p.183]{Ratcliffe}, and so also ($F_3$), so $W$ may be taken to be a regular neighbourhood of an embedding of a finite 3-skeleton of $BG$ into $\bR^6$. As further examples of such manifolds, one may take
\begin{equation}\label{ex:From3Mfld}
W=(M^3 \times D^3) \# g(S^3 \times S^3)
\end{equation}
where $M^3$ is a closed oriented 3-manifold that is irreducible (so that $\pi_2(M)=0$) and has virtually polycyclic fundamental group.

Recall from \S\ref{sec:pi1} that the Hirsch length of a virtually polycyclic group is the number of infinite cyclic quotients in a subnormal series.

\begin{theorem}\label{thm:NinasEx}
Suppose that $G$ is virtually polycyclic of Hirsch length $h$, and $W$ is a 6-manifold satisfying (i) and (ii) above, of genus $g(W) \geq 7 + h$. Then there is a short exact sequence
\begin{equation*}
0 \lra G^{ab} \lra \Gamma_\partial(W)^{ab} \lra \mathrm{ko}_7(BG) \lra 0
\end{equation*}
which is (non-canonically) split.
\end{theorem}

This short exact sequence was first established by Friedrich \cite{NinaThesis}, who also showed that it is split after inverting 2. We shall give a different argument to hers, which gives the splitting at the prime 2 as well.

The following three examples concern the manifolds $W$ of construction \eqref{ex:From3Mfld} with $M$ a 3-manifold having finite fundamental group, and $g \geq 7$ so the hypotheses of Theorem \ref{thm:NinasEx} are satisfied.

\begin{example}\label{ex:Ab1}
Let $M^3 = L_{p,q}^3$ be the $(p,q)$th lens space, with fundamental group $G=\bZ/p$ with $p$ prime. Then we have
$$\Gamma_\partial(W)^{ab} \cong \begin{cases}
\bZ/2 \oplus \bZ/4 & \text{if $p=2$}\\
\bZ/3 \oplus \bZ/9 & \text{if $p=3$}\\
(\bZ/p)^3 & \text{if $p\geq 5$}.
\end{cases}$$

The required calculation of $\mathrm{ko}_7(B\bZ/p)$ may be extracted from \cite[Example 7.3.1]{GB2} for $p=2$, and from the Atiyah--Hirzebruch spectral sequence, the fact that $\mathrm{ko}_*(B\bZ/p)[\tfrac{1}{2}]$ is a summand of $\mathrm{ku}_*(B\bZ/p)[\tfrac{1}{2}]$, and \cite[Remark 3.4.6]{GB1} for odd $p$.
\end{example}

\begin{example}
Let $M^3$ be the spherical 3-manifold with fundamental group $G=Q_8$. Then we have
$$\Gamma_\partial(W)^{ab} \cong (\bZ/2)^2 \oplus (\bZ/4)^2 \oplus \bZ/64.$$

The required calculation of $\mathrm{ko}_7(BQ_8)$ may be extracted from \cite[p.\ 138]{GB2}.
\end{example}

\begin{example}
Let $M^3 = \Sigma^3$ be the Poincar{\'e} homology 3-sphere, with fundamental group $G$ the binary icosahedral group, isomorphic to $SL_2(\bF_5)$. Then $g$ we have
$$\Gamma_\partial(W)^{ab} \cong (\bZ/5)^2 \oplus \bZ/9 \oplus \bZ/64.$$

As $\Sigma^3$ is a homology sphere, $H_1(BG;\bZ) \cong H_1(\Sigma^3;\bZ)=0$ so we must just calculate $\mathrm{ko}_7(BG)$. The order of $G \cong SL_2(\bF_5)$ is $120 = 2^3 \cdot 3 \cdot 5$, so we shall calculate the localisations $\mathrm{ko}_7(BG)_{(p)}$ for $p \in \{2,3,5\}$. Recall that the cohomology ring of $BG$ is $H^*(BG;\bZ) = \bZ[z]/(120 \cdot z)$ with $|z|=4$, which may be computed from the fibration sequence $S^3 \to \Sigma^3 \to BG$.

When $p$ is odd, $G$ has cyclic Sylow $p$-subgroup, so by transfer $\mathrm{ko}_7(BG)_{(p)}$ is a summand of $\mathrm{ko}_7(B\bZ/p)_{(p)}$, which we have explained in Example \ref{ex:Ab1} is $\bZ/9$ for $p=3$ and $(\bZ/5)^2$ for $p=5$. Comparing this with the Atiyah--Hirzebruch spectral sequence computing $\mathrm{ko}_*(BG)_{(p)}$ gives the claimed answer.

When $p=2$, $G$ has Sylow 2-subgroup $Q_8$, so by transfer $\mathrm{ko}_7(BG)_{(2)}$ is a summand of $\mathrm{ko}_7(BQ_8)_{(2)}$, which we have explained in Example \ref{ex:Ab1} is $(\bZ/4)^2 \oplus \bZ/64$. By a theorem of Mitchell and Priddy \cite[Theorem D]{MitchellPriddy} there is a stable splitting of $BQ_8$ as $BG_{(2)} \vee X \vee X$ for some spectrum $X$, from which it follows that $\mathrm{ko}_7(BG)_{(2)}$ is either $\bZ/64$ or $(\bZ/4)^2 \oplus \bZ/64$; we may see that the first case occurs from the Atiyah--Hirzebruch spectral sequence computing $\mathrm{ko}_*(BG)_{(2)}$.
\end{example}

The rest of this section is concerned with the proof of Theorem~\ref{thm:NinasEx}.

\subsubsection{Reduction to homotopy theory}

We have assumed that $(W, \partial W)$ is 2-connected, so by Lemma \ref{lem:aHutContr} we have that $\hAut(\theta, {\partial W}) \simeq *$. Thus by Theorem \ref{thm:homotopical} and the discussion in \S\ref{sec:pi1} there is a map
$$\alpha : \mathcal{M}(W) \simeq \mathcal{M}^{\theta}(W) \lra \Omega^\infty MT\theta = \Omega^\infty( MT\Spin(6) \wedge BG_+)$$
which induces an isomorphism on homology onto the path component that it hits in degrees $* \leq \frac{g(W)-h-5}{2}$, so in particular as long as $g(W) \geq 7 + h$ it induces an isomorphism on first homology. As described above the first homology of $\mathcal{M}(W)$ is the abelianisation of the mapping class group $\Gamma_\partial(W)$, so to establish Theorem \ref{thm:NinasEx} we must establish a short exact sequence
\begin{equation}\label{eq:ToSplit2}
0 \lra G^{ab} \lra \pi_1(MT\Spin(6) \wedge BG_+) \lra \mathrm{ko}_7(BG) \lra 0
\end{equation}
and show that it is split. 

\subsubsection{The exact sequence}

Specialising the construction in \S\ref{sec:AbWgs} to $2n=6$, we showed that there is a cofibration sequence of spectra
\begin{equation}\label{eq:CofSeq}
F \lra MT\Spin(6) \lra \Sigma^{-6}M\Spin,
\end{equation}
that $F$ is connective, and that $\pi_0(F) \cong \bZ$ and $\pi_1(F) \cong \bZ/4$. The Atiyah--Hirzebruch spectral sequence for $F \wedge BG_+$ gives isomorphisms
$$\pi_0(F \wedge BG_+) \cong \bZ \quad \quad\quad \pi_1(F \wedge BG_+) \cong \bZ/4 \oplus H_1(BG;\bZ)$$
where the splitting in the latter is induced by the retraction $S^0 \to BG_+ \to S^0$ of pointed spaces. Smashing the cofibration sequence \eqref{eq:CofSeq} with $BG_+$, the associated long exact sequence of homotopy groups has the form
$$\Omega_8^{\Spin}(BG) \overset{\partial}\lra \bZ/4 \oplus H_1(BG;\bZ) \lra \pi_1(MT\Spin(6) \wedge BG_+) \lra \Omega_7^{\Spin}(BG) \lra \bZ $$
and by naturality the long exact sequence for $G=\{e\}$ splits off of this one. As $\pi_1(MT\Spin(6))=0$ by Theorem \ref{thm:AbWgsThm}, and $\Omega_7^{\Spin}=0$ by \cite{MilnorSpin}, by the discussion in \S\ref{sec:AbWgs} this leaves an exact sequence
$$\widetilde{\Omega}_8^{\Spin}(BG) \overset{\partial}\lra H_1(BG;\bZ) \lra \pi_1(MT\Spin(6) \wedge BG_+) \lra \Omega_7^{\Spin}(BG) \lra 0.$$
The Atiyah--Bott--Shapiro map $M\Spin \to ko$ is 8-connected, so the induced map $\Omega_7^{\Spin}(BG) \to \mathrm{ko}_7(BG)$ is an isomorphism. To obtain the claimed short exact sequence we shall therefore prove the following.

\begin{lemma}\label{lem:ConnectingTrivial}
If $G$ is finitely-generated then the connecting map $\partial : \widetilde{\Omega}_8^{\Spin}(BG) \to H_1(BG;\bZ)$ is trivial.
\end{lemma}
\begin{proof}
If this connecting map were non-trivial for some finitely-generated $G$, then because every non-trivial element of $H_1(BG;\bZ) = G^{ab}$ remains non-trivial under some homomorphism $G^{ab} \to \bZ/p^k$ with $p$ prime and $k \geq 1$, by naturality this connecting map would be non-trivial for $G=\bZ/p^k$. So it suffices to show that the map is trivial in this case.

As $M\Spin \to ko$ is 8-connected, the map $\widetilde{\Omega}_8^{\Spin}(BG) \to \widetilde{\mathrm{ko}}_8(BG)$ is an isomorphism. The result then follows as $\widetilde{\mathrm{ko}}_8(B\bZ/p^k)=0$, by a trivial application of the Atiyah--Hirzebruch spectral sequence for $p$ odd and by \cite[Theorem 2.4]{BGS} for $p=2$.
\end{proof}

\subsubsection{Simplifying the splitting problem}

Let $x : F \to H\bZ$ be the 0-th Postnikov truncation of $F$. The composition
$$k: \Sigma^{-6} M\Spin \overset{\partial}\lra \Sigma F \overset{\Sigma x}\lra \Sigma H\bZ$$
represents, under the Thom isomorphism, some element $\kappa \in H^7(B\Spin;\bZ)$ which we identify as follows.

\begin{lemma}
We have $\kappa = \beta(w_6)$.
\end{lemma}
\begin{proof}
We have $H^7(B\Spin;\bZ/2) = \bZ/2\{w_7\}$, but $w_7$ of course vanishes when restricted to $B\Spin(6)$. The long exact sequence on $\bZ/2$-cohomology for \eqref{eq:CofSeq} contains the portion
$$H^1(MT\Spin(6);\bZ/2) \longleftarrow H^1(\Sigma^{-6}M\Spin;\bZ/2) = \bZ/2\{w_7 \cdot u_{-6}\} \longleftarrow H^1(\Sigma F ;\bZ/2)$$
and the left-hand map is zero, so the right-hand map is surjective. As the map $x$ is 1-connected, the map $\Sigma x$ is 2-connected, so we have an isomorphism
$$(\Sigma x)^* : \bZ/2\{\Sigma \iota\} = H^1(\Sigma H\bZ;\bZ/2) \overset{\sim}\lra H^1(\Sigma F;\bZ/2).$$
It follows that $\kappa$ reduces to $w_7 \neq 0$ modulo 2. The integral cohomology of $B\Spin$ is known to only have torsion of order 2 (this may be deduced from \cite{Kono}), so the Bockstein sequence for $B\Spin$ shows that $H^7(B\Spin;\bZ) = \bZ/2$, which must therefore be generated by $\beta w_6$ as this reduces modulo 2 to $\mathrm{Sq}^1(w_6)=w_7$. Thus $\kappa = \beta(w_6)$.
\end{proof}
As $w_6 \cdot u_{-6} = \mathrm{Sq}^6(u_{-6})$, we may write the map $k$ as the composition
$$k: \Sigma^{-6} M\Spin \overset{u_{-6}}\lra \Sigma^{-6} H\bZ/2 \overset{\mathrm{Sq}^6}\lra H\bZ/2 \overset{\beta}\lra \Sigma H\bZ$$
and so we may form a spectrum $E$ fitting into the diagram
\[\begin{tikzcd} 
MT\Spin(6) \rar \dar& \Sigma^{-6} M\Spin \rar{\partial} \dar{u_{-6}}& \Sigma F \dar{\Sigma x}\\
E \rar& \Sigma^{-6} H\bZ/2 \rar{\beta\mathrm{Sq}^6}& \Sigma H\bZ
\end{tikzcd}\]
in which the rows are homotopy cofibre sequences. Smashing with $BG_+$ and considering the map of long exact sequences gives
\[\begin{tikzcd} 
\Omega_8^{\Spin}(BG) \rar \dar & H_8(BG;\mathbb{Z}/2) \dar{(\beta\mathrm{Sq}^6)_*=0}\\
\bZ/4 \oplus H_1(BG;\mathbb{Z}) \rar{\mathrm{pr}_2} \dar&  H_1(BG;\mathbb{Z}) \dar\\
\pi_1(MT\Spin(6)\wedge BG_+) \rar \dar& \pi_1(E\wedge BG_+) \dar \arrow[bend right=40, dashed, swap]{u}{\sigma_G}\\
\Omega_7^{\Spin}(BG) \rar \dar& H_7(BG;\mathbb{Z}/2) \dar{(\beta\mathrm{Sq}^6)_*=0}\\
0 \rar& H_0(BG;\mathbb{Z})
\end{tikzcd}\]
where the columns are exact and the two indicated maps are zero by instability of the (co)homology operation $\beta\mathrm{Sq}^6$ in these degrees. It therefore suffices to show that the right-hand short exact sequence has a dashed splitting $\sigma_G$ as indicated. (Note that the commutativity of the top square gives another proof of Lemma \ref{lem:ConnectingTrivial}.)

\subsubsection{Reducing the splitting problem to cyclic groups}

The abelianisation homomorphism $a: G \to G^{ab}$ induces a map on the short exact sequences of the form
\[\begin{tikzcd} 
0 \rar \dar& H_1(BG;\mathbb{Z}) \rar \dar{=} & \pi_1(E \wedge BG_+) \dar{a_*} \rar & H_7(BG;\mathbb{Z}/2) \dar{a_*} \rar& 0 \dar\\
0 \rar & H_1(BG^{ab};\mathbb{Z}) \rar & \arrow[l, bend right=30, dashed, "\sigma_{G^{ab}}"] \pi_1(E \wedge BG^{ab}_+) \rar & H_7(BG^{ab};\mathbb{Z}/2) \rar& 0
\end{tikzcd}\]
so if we can show that the short exact sequence is split for $G^{ab}$, say via the dashed homomorphism $\sigma_{G^{ab}}$, then the sequence for $G$ is also split, via $\sigma_G := \sigma_{G^{ab}} \circ a_*$. This reduces us to studying abelian groups.

On the other hand, if such short exact sequences are split for each cyclic group of prime power order or $\bZ$, then we may write $G^{ab} = C_1 \oplus \cdots \oplus C_n$ for cyclic groups $C_i$ of prime power order or $\bZ$ and combine their splittings to obtain one for $G^{ab}$ (though of course it depends on the choice of expression for $G^{ab}$ as a sum of cyclic subgroups, so is not canonical). We have therefore reduced the question of splitting the short exact sequences \eqref{eq:ToSplit2} to the case of such cyclic groups.

\subsubsection{The splitting for $G=\mathbb{Z}/p^k$ or $\bZ$}

If $G=\bZ$ or $G=\bZ/p^k$ with $p$ odd then $H_7(BG;\bZ/2)=0$ and so the short exact sequence becomes
$$0 \lra G \lra \pi_1(E \wedge BG_+) \lra 0 \lra 0$$
which is certainly split. 

For $G=\bZ/2^k$ we must go to more trouble: in this case the short exact sequence becomes
\begin{equation}\label{eq:Z2kSeq}
0 \lra \bZ/2^k \lra \pi_1(E \wedge B\bZ/2^k_+) \lra \bZ/2 \lra 0
\end{equation}
so is either split or else $\pi_1(E \wedge B\bZ/2^k_+) \cong \bZ/2^{k+1}$. Consider first smashing the cofibre sequence defining $E$ with $S/2$, giving the cofibration sequence
$$S/2 \wedge E \lra S/2 \wedge \Sigma^{-6} H\bZ/2 \simeq \Sigma^{-6} H\bZ/2 \vee \Sigma^{-5} H\bZ/2 \overset{\mathrm{Sq}^7 \vee \mathrm{Sq}^6}\lra \Sigma H\bZ/2 \simeq S/2 \wedge H\bZ.$$
Now further smashing this with $B\bZ/2^k_+$ and considering the long exact sequence on homotopy groups gives an exact sequence
\begin{equation*}
\begin{tikzcd}
       H_8(B\bZ/2^k;\bZ/2) \oplus H_7(B\bZ/2^k;\bZ/2) \rar{\mathrm{Sq}^7_* \oplus \mathrm{Sq}^6_*} \ar[draw=none]{d}[name=X, anchor=center]{} & H_1(B\bZ/2^k;\bZ/2)
                \ar[rounded corners,
                to path={ -- ([xshift=2ex]\tikztostart.east)
                	|- (X.center) \tikztonodes
                	-| ([xshift=-2ex]\tikztotarget.west)
                	-- (\tikztotarget)}]{dl} \\
              \pi_1(S/2 \wedge E\wedge B\bZ/2^k_+) \rar \ar[draw=none]{d}[name=Y, anchor=center]{}& H_7(B\bZ/2^k;\bZ/2) \oplus H_6(B\bZ/2^k;\bZ/2) 
							\ar[rounded corners,
                to path={ -- ([xshift=2ex]\tikztostart.east)
                	|- (Y.center) \tikztonodes
                	-| ([xshift=-2ex]\tikztotarget.west)
                	-- (\tikztotarget)}]{dl}[near end]{\mathrm{Sq}^7_* \oplus \mathrm{Sq}^6_*}\\
						  H_0(B\bZ/2^k;\bZ/2)
\end{tikzcd}
\end{equation*}
and so $\pi_1(S/2 \wedge E\wedge B\bZ/2^k_+)$ has cardinality 8, because the homology operations $\mathrm{Sq}^7_*$ and $\mathrm{Sq}^6_*$ are zero in these degrees by instability. On the other hand computing with the long exact sequence as above gives an exact sequence
$$0 \lra \bZ = H_0(B\bZ/2^k ; \bZ) \lra \pi_0(E \wedge B\bZ/2^k_+) \lra H_6(B\bZ/2^k ; \bZ/2) = \bZ/2 \lra 0$$
which is split via $B\bZ/2^k_+ \to S^0$. Thus
$$\mathrm{Ker}(2 \cdot - : \pi_0(E \wedge B\bZ/2^k_+) \to \pi_0(E \wedge B\bZ/2^k_+)) = \bZ/2$$
and so, as $\pi_1(S/2 \wedge E\wedge B\bZ/2^k_+)$ has cardinality 8, it follows that
$$\mathrm{Coker}(2 \cdot - : \pi_1(E \wedge B\bZ/2^k_+) \to \pi_1(E \wedge B\bZ/2^k_+))$$
has cardinality 4. Thus $\pi_1(E \wedge B\bZ/2^k_+)$ cannot be cyclic, so \eqref{eq:Z2kSeq} is split.

\vspace{1ex}

\noindent \textbf{Acknowledgements.} The authors would like to thank M.\ Krannich for useful comments on a draft of this paper. S.\ Galatius was partially supported by the European Research Council (ERC) under the European Union's Horizon 2020 research and innovation programme (grant agreement No.\ 682922), and by NSF grant DMS-1405001. O.\ Randal-Williams was partially supported by the ERC under the European Union's Horizon 2020 research and innovation programme (grant agreement No.\ 756444).


\bibliographystyle{amsalpha}
\bibliography{../biblio}

\end{document}